\documentclass[12pt,a4paper,reqno]{amsart}
\usepackage[T1]{fontenc}
\usepackage{amsmath,amssymb,ifthen}
\usepackage{fullpage}
\usepackage{graphicx,psfrag,subfigure}
\usepackage[dvips]{color}

\def\H{\widetilde{H}}

\def\R{{\mathbb R}}
\def\N{{\mathbb N}}

\def\EE{{\mathcal E}}
\def\HH{{\mathcal H}}
\def\MM{{\mathcal M}}
\def\NN{{\mathcal N}}

\def\PP{{\mathcal P}}
\def\RR{{\mathcal R}}
\def\SS{{\mathcal S}}
\def\TT{{\mathcal T}}
\def\UU{\mathcal U}
\def\XX{{\mathcal X}}

\def\diam{{\rm diam}}
\def\edual#1#2{\langle\hspace*{-1mm}\langle#1\,,\,#2\rangle\hspace*{-1mm}\rangle}

\def\norm#1#2{\|#1\|_{#2}}

\def\set#1#2{\big\{#1\,:\,#2\big\}}
\def\Crel{C_{\rm rel}}
\def\CCrel{\widetilde C_{\rm rel}}
\def\Ceff{C_{\rm eff}}

\def\Cinv{C_{\rm inv}}
\def\eps{\varepsilon}

\def\U{\mathbf{U}}

\def\slp{\mathfrak{V}} 
\def\dlp{\mathfrak{K}} 

\newcommand{\enorm}[2][]{#1|\hspace*{-.5mm}#1|\hspace*{-.5mm}#1|#2#1|\hspace*{-.5mm}#1|\hspace*{-.5mm}#1|}

\def\normL2#1#2{\|#1\|_{L^2(#2)}}

\newcommand{\dual}[3][]{#1\langle#2\,,\,#3#1\rangle}

\newcounter{constantsnumber}
\def\namec#1#2{%
 \ifthenelse{\equal{#1}{rel}}{C_{\rm rel}}{%
  \ifthenelse{\equal{#1}{mesh}}{C_{\rm mesh}}{%
  \ifthenelse{\equal{#1}{sz}}{C_{\rm sz}}{%
  \ifthenelse{\equal{#1}{dislocrel}}{C_{\rm dlr}}{%
  \ifthenelse{\equal{#1}{eff}}{C_{\rm eff}}{%
  \ifthenelse{\equal{#1}{main}}{C_{\rm V}}{%
  \ifthenelse{\equal{#1}{opt}}{C_{\rm opt}}{%
  \ifthenelse{\equal{#1}{normequiv}}{C_{\rm norm}}{%
  \ifthenelse{\equal{#1}{reliable}}{C_{\rm rel}}{%
  \ifthenelse{\equal{#1}{efficient}}{C_{\rm eff}}{%
  \ifthenelse{\equal{#1}{dlr}}{C_{\rm dlr}}{%
  \ifthenelse{\equal{#1}{reduction}}{C_{\rm red}}{%
    \ifthenelse{\equal{#1}{locall}}{C_{\rm loc}}{%
   \ifthenelse{\equal{#1}{unibound}}{C_{\rm hot}}{%
    \ifthenelse{\equal{#1}{hotConst}}{C_{\rm hot}}{%
   \ifthenelse{\equal{#1}{inverseK}}{C_{\rm K}}{%
  \ifthenelse{\equal{#1}{refined}}{C_{\rm ref}}{%
  \ifthenelse{\equal{#1}{estconv}}{C_{\rm est}}{%
    \ifthenelse{\equal{#1}{meshsize}}{C_{\rm mesh}}{%
  \ifthenelse{\equal{#1}{optimal}}{C_{\rm opt}}{%
  \ifthenelse{\equal{#1}{qo}}{C_{\rm qo}}{%
    \ifthenelse{\equal{#1}{elliptic}}{C_{\rm ell}}{%
  \ifthenelse{\equal{#1}{mon}}{C_{\rm mon}}{%
  \ifthenelse{\equal{#1}{cea}}{C_{\mbox{\rm\scriptsize C\'ea}}}{%
  \ifthenelse{\equal{#2}{newcounter}}{\refstepcounter{constantsnumber}\label{const#1}}{}C_{\ref{const#1}}}%
}}}}}}}}}}}}}}}}}}}}}}}}
\def\setc#1{\namec{#1}{newcounter}}
\def\c#1{\namec{#1}{reference}}

\newcounter{contractionnumber}
\def\nameq#1#2{%
  \ifthenelse{\equal{#1}{reduction}}{q_{\rm red}}{%
  \ifthenelse{\equal{#1}{estconv}}{q_{\rm est}}{%
  \ifthenelse{\equal{#1}{cea}}{q_{\mbox{\scriptsize C\'ea}}}{%
  \ifthenelse{\equal{#2}{newcounter}}{\refstepcounter{contractionnumber}\label{contraction#1}}{}q_{\ref{contraction#1}}}%
}}}

\def\namer#1#2{%
  \ifthenelse{\equal{#1}{reduction}}{\rho_{\rm red}}{%
  \ifthenelse{\equal{#1}{estconv}}{\rho_{\rm est}}{%
  \ifthenelse{\equal{#1}{cea}}{\rho_{\mbox{\scriptsize C\'ea}}}{%
  \ifthenelse{\equal{#1}{qo}}{\rho_{\mbox{\scriptsize qo}}}{%
  \ifthenelse{\equal{#2}{newcounter}}{\refstepcounter{contractionnumber}\label{contraction#1}}{}\rho_{\ref{contraction#1}}}%
}}}}

\def\ffor{\quad\text{for all }}
\def\k{\kappa}

\def\semiHs#1#2{|#1|_{H^{s}(#2)}}
\def\semiH#1#2{|#1|_{H^{1}(#2)}}
\def\semiHh#1#2{|#1|_{H^{1/2}(#2)}}

\def\patchref#1{\widehat{\omega}_{#1}}

\def\nodes{\mathcal{N}_{\star}}
\def\mesh{\mathcal{T}_{\star}}
\def\patch{\omega_{z}}

\def\sgrad#1{\nabla_{#1}}

\newtheorem{theorem}{Theorem}
\newtheorem{proposition}[theorem]{Proposition}
\newtheorem{lemma}[theorem]{Lemma}

\newtheorem{algorithm}[theorem]{Algorithm}

\newenvironment{remark}{\refstepcounter{theorem}\medskip\noindent\textbf{Remark~\thetheorem.}\ \it}{\qed\smallskip}

\def\subsection#1{\bigskip

\refstepcounter{subsection}{\bf\thesubsection.~#1.~~}}

\def\subsubsection#1{\bigskip

\refstepcounter{subsubsection}{\bf\thesubsubsection.{~#1.}~~}}

\title[Convergence of ABEM and adaptive FEM-BEM coupling]
{Convergence of adaptive BEM and adaptive FEM-BEM coupling
for estimators without $h$-weighting factor}

\author{Michael Feischl}
\author{Thomas F\"uhrer}
\author{Gregor Mitscha-Eibl}
\author{\\Dirk Praetorius}

\address{Institute for Analysis and Scientific Computing,
      Vienna University of Technology,
      Wiedner Hauptstra\ss{}e 8-10,
      A-1040 Wien, Austria}
      
\email{\{\,Michael.Feischl\,,\,Thomas.Fuehrer\,,\,Gregor.Mitscha-Eibl\,\}@tuwien.ac.at}
\email{Dirk.Praetorius@tuwien.ac.at \rm{(corresponding author)}}

\author{Ernst P.~Stephan}

\address{Institute for Applied Mathematics,
Leibniz University Hannover,
Welfengarten 1,
D-30167 Hannover, Germany}
\email{stephan@ifam.uni-hannover.de}

\date{\today}

\keywords{boundary element method (BEM), FEM-BEM coupling, a~posteriori error estimate, adaptive algorithm, convergence}
\subjclass[2000]{65N12, 65N38, 65N30, 65N50}

\begin{document}
\begin{abstract}
We analyze adaptive mesh-refining algorithms in the frame of boundary 
element methods (BEM) and the coupling of finite elements and boundary 
elements (FEM-BEM). Adaptivity is driven by 
the two-level error 
estimator proposed by Ernst P. Stephan, Norbert Heuer, and coworkers in 
the frame of BEM and FEM-BEM or by the residual error estimator introduced 
by Birgit Faermann for BEM for weakly-singular integral equations. We prove that in either case
the usual adaptive algorithm drives the associated error estimator to
zero. Emphasis is put on the fact that the error estimators considered
are \emph{not even globally} equivalent to weighted-residual error 
estimators for which recently convergence with quasi-optimal algebraic
rates has been derived. 
\end{abstract}
\maketitle

\def\Cmon{C_{\rm mon}}
\def\Ccont{C_{\rm cont}}
\section{Introduction}
\label{section:intro}

\noindent
A~posteriori error estimation and related adaptive mesh-refining algorithms
are one important basement of modern scientific computing. Starting from an
initial mesh $\TT_0$ and based on a computable a~posteriori error estimator,
such algorithms iterate the loop
\begin{align}\label{eq:semr}
 \boxed{\texttt{~solve~}}
 \quad\to\quad
 \boxed{\texttt{~estimate~}}
 \quad\to\quad
 \boxed{\texttt{~mark~}}
 \quad\to\quad
 \boxed{\texttt{~refine~}}
\end{align}
to create a sequence of successive locally refined meshes $\TT_\ell$,
corresponding discrete solutions $U_\ell$, as well as a~posteriori error
estimators $\mu_\ell$. We consider the frame of conforming Galerkin 
discretizations, where $\TT_\ell$ is linked to a finite-dimensional subspace 
$\XX_\ell$ of a Hilbert space $\HH$ with corresponding Galerkin solution
$U_\ell\in\XX_\ell$, where successive refinement guarantees nestedness
$\XX_\ell\subseteq\XX_{\ell+1}\subset\HH$ for all $\ell\in\N_0$. 

Convergence of this type of adaptive algorithm in the sense of 
\begin{align}
 \lim_{\ell\to\infty}\norm{u-U_\ell}\HH = 0
\end{align}
has first been addressed in~\cite{bv} for 1D FEM and~\cite{doerfler} 
for 2D FEM.
We note
that already the pioneering work~\cite{bv} observed that 
validity of some C\'ea-type quasi-optimality and nestedness
$\XX_\ell\subseteq\XX_{\ell+1}$ for all $\ell\in\N_0$ imply 
a~priori convergence
\begin{align}\label{eq:aprioriconv}
 \lim_{\ell\to\infty}\norm{U_\infty-U_\ell}\HH = 0,
\end{align}
where $U_\infty$ is the unique Galerkin solution
in $\XX_\infty:=\overline{\bigcup_{\ell\in\N_0}\XX_\ell}$.
From a conceptual point of view, it thus only remained to identify the limit
$u=U_\infty$.
Based on such an a~priori convergence result~\eqref{eq:aprioriconv}, a general theory of convergence 
of adaptive FEM is devised in~\cite{msv,siebert}, where the analytical focus
is on \emph{estimator convergence}
\begin{align}\label{eq:convergence}
 \lim_{\ell\to\infty}\mu_\ell = 0.
\end{align}
Moreover, the recent work~\cite{axioms} gives an analytical frame to guarantee convergence with optimal convergence rates;
see also the overview article~\cite{arcme} for the current state of the art of 
adaptive BEM.
Throughout, it is however implicitly assumed that the local contributions 
$\mu_\ell(T)$ of the error estimator $\mu_\ell$ are weighted with the 
local mesh-size, i.e., $|T|^{\alpha}$ for some appropriate $\alpha>0$, or that $\mu_\ell$ is \emph{locally} equivalent to a mesh-size weighted error estimator.

In this work, we consider two particular error estimators
whose local contributions are not weighted by the local mesh-size. 
We devise
a joint analytical frame which proves estimator convergence~\eqref{eq:convergence}.  First, we let $\mu_\ell$ be the Faermann error estimator~\cite{faermann2d,faermann3d,cf} for BEM for the weakly-singular 
integral equation with 
$\HH=\H^{-1/2}(\Gamma)$. The local contributions of $\mu_\ell$ are 
overlapping $H^{1/2}$-seminorms of the residual 
$F-AU_\ell \in H^{1/2}(\Gamma)$. The striking point of $\mu_\ell$ is that it is 
the only a~posteriori BEM error estimator which is known to be both reliable and
efficient without any further assumptions on the given data, i.e., it holds
\begin{align}
 \Ceff^{-1}\,\mu_\ell \le \norm{u-U_\ell}{\HH} \le \Crel\,\mu_\ell
\end{align}
with $\ell$-independent constants $\Ceff,\Crel>0$. We note that $\mu_\ell$
is not equivalent to an $h$-weighted error estimator which prevents to follow
the arguments from the available literature.
 
Second, our analysis covers the two-level error estimators for 
BEM~\cite{msw,mms,ms,hms,heuer,eh} or the adaptive FEM-BEM 
coupling~\cite{ms:fembem,kms,gms,afkp:fembem}. The local contributions are
projections of the computable error between two Galerkin solutions
onto one-dimensional spaces, spanned by hierarchical basis functions.
These estimators are known to be
efficient. On the other hand, reliability is only proven under an appropriate
saturation assumption which is even equivalent to reliability for the 
symmetric BEM operators~\cite{effp,efgp,hypsing3d}. However, such a saturation
assumption is formally equivalent to asymptotic convergence of the adaptive algorithm
\cite{fp} which cannot be guaranteed mathematically in general and is 
expected to fail on coarse meshes.

\textbf{Outline.}
The remainder of the paper is organized as follows: In Section~\ref{section:abstract}, 
we introduce an abstract frame which covers both BEM as well as the FEM-BEM coupling.
We formally state the adaptive loop (Algorithm~\ref{algorithm}). Under three
assumptions on the error estimator which are later verified for the particular
model problems, we prove that the adaptive loop drives the underlying error estimator
to zero (Proposition~\ref{prop:convergence} and Proposition~\ref{prop:new}). 
Section~\ref{section:symm} treats the weakly-singular integral equation associated
with the Laplacian. We prove that two-level error estimator 
(Theorem~\ref{thm:symm:twolevel}) as well as Faermann error estimator 
(Theorem~\ref{thm:faermann-convergence}) fit into the abstract framework.
In Section~\ref{section:hypsing}, we consider the hyper-singular integral equation
associated with the Laplacian. We prove that the two-level error estimator
fits into the abstract framework (Theorem~\ref{thm:hypsing:twolevel}).
The final Section~\ref{section:fembem} considers a nonlinear Laplace transmission
problem which is reformulated by some FEM-BEM coupling. We prove that the 
two-level error estimator fits into the abstract framework as well
(Theorem~\ref{thm:fembem:twolevel}).

\textbf{Notation.}
Associated quantities are linked through the same index, i.e., $U_\star$ is the 
discrete solution with respect to the discrete space $\XX_\star$ which corresponds 
to the triangulation $\TT_\star$. Throughout, the star is understood as general 
index and may be accordingly replaced by the level of the adaptive algorithm 
(e.g., $U_\ell$) or by the infinity symbol (e.g., $\XX_\infty$).
All constants as well as their dependencies are explicitly given in statements and  
results. In proofs, we shall use $A\lesssim B$ to abbreviate $A\le c\,B$ with 
some generic multiplicative constant $c>0$ which is clear from the context. 
Moreover, $A\simeq B$ abbreviates $A\lesssim B \lesssim A$.

\section{Abstract setting}
\label{section:abstract}

\subsection{Model problem}
Let $\HH$ be a Hilbert space with dual space $\HH^*$ and $A:\HH\to\HH^*$ be 
a bi-Lipschitz continuous operator, i.e.,
\begin{align}\label{eq:cont}
 \Ccont^{-1}\,\norm{w-v}\HH\leq \norm{Aw-Av}{\HH^*} \le \Ccont \, \norm{w-v}\HH
\end{align}
for all $v,w\in\HH$. Here, $\norm\cdot{\HH^*}$ denotes the operator norm on 
$\HH^*$,
\begin{align}
 \norm{F}{\HH^*} = \sup_{v\in\HH\backslash\{0\}}\frac{|\dual{F}{v}|}{\norm{v}\HH}
 \quad\text{for all }F\in\HH^*.
\end{align}
Suppose that there exists some subspace $\XX_{00}\subseteq \HH$ such that for any given closed subspace $\XX_{00}\subseteq \XX_\star\subseteq \HH$ and any
continuous linear functional $F\in\HH^*$ on $\HH$, the Galerkin formulation
\begin{align}\label{eq:galerkin}
 \dual{AU_\star}{V_\star} 
 = \dual{F}{V_\star}
 \quad\text{for all }V_\star\in\XX_\star
\end{align}
admits a unique solution $U_\star\in\XX_\star$, where $\dual\cdot\cdot$ denotes the duality bracket between $\HH$ and its dual $\HH^*$.
Particularly, this implies the existence of a unique solution $u\in\HH$ of
\begin{align}\label{eq:continuous}
 A u = F.
\end{align}
 Moreover, we suppose that there holds the 
C\'ea-type estimate
\begin{align}\label{eq:cea}
 \norm{u-U_\star}\HH
 \le \c{cea}\,\min_{V_\star\in\XX_\star} \norm{u-V_\star}\HH,
\end{align}
where the constant $\c{cea}>0$ depends only on the operator $A$ (and possibly on $F$). 
To be precise,
we will write $u=u(F)$ and $U_\star = U_\star(F)$ in the following to indicate
that $u(F)$ resp.\ $U_\star(F)$ are the unique solutions with respect to some
given right-hand side $F\in\HH^*$.

\begin{remark}
{\rm (i)} The assumptions~\eqref{eq:cont}--\eqref{eq:cea} are particularly satisfied with $\XX_{00}=\{0\}$, $\Ccont=\max\{\widetilde C_{\rm cont},\Cmon^{-1}\}$, and $\c{cea}=\widetilde C_{\rm cont}/\Cmon$ if $A$ is Lipschitz continuous and strongly monotone in the sense
 \begin{align}\label{eq:cmon}
 \norm{Aw-Av}{\HH^*} \le \widetilde C_{\rm cont} \, \norm{w-v}\HH\quad\text{and}\quad\Cmon \norm{w-v}{\HH}^2\leq \dual{Aw-Av}{w-v}
 \end{align}
 for all $v,w\in\HH$; see e.g.~\cite[Section~25.4]{zeidler} for the corresponding proofs. In particular, this also covers linear problems in the frame of the Lax-Milgram lemma, e.g., the symmetric BEM formulations of Section~\ref{section:symm}--\ref{section:hypsing}.
 
 {\rm (ii)} The assumptions~\eqref{eq:cont}--\eqref{eq:cea} are motivated by the FEM-BEM coupling formulations in Section~\ref{section:fembem}.
 
 {\rm (iii)} For $A$ being linear, it is also sufficient if additionally to~\eqref{eq:cont}, $A$ satisfies a uniform inf-sup-condition along the sequence of discrete subspaces $\XX_\ell$ generated by Algorithm~\ref{algorithm} below.
\end{remark}

\subsection{Adaptive algorithm}
We shall assume that $\XX_\ell$ is a finite-dimensional subspace of $\HH$ 
related to some triangulation $\TT_\ell$ and that $U_\ell(F)\in\XX_\ell$ is 
the corresponding Galerkin solution~\eqref{eq:galerkin} for $\XX_\star=\XX_\ell$. Starting from an initial mesh $\TT_0$, the 
triangulations $\TT_\ell$ are 
successively refined by means of the following realization of~\eqref{eq:semr}, where 
\begin{align}
 \mu_\ell(F) := \mu_\ell(F;\TT_\ell)
 \quad\text{with}\quad
 \mu_\ell(F;\EE_\ell):=\Big(\sum_{T\in\EE_\ell}\mu_\ell(F;T)^2\Big)^{1/2}<\infty
 \quad\text{for all }\EE_\ell\subseteq\TT_\ell
\end{align}
is a computable a~posteriori error estimator. Its local contributions $\mu_\ell(F;T)\ge0$
measure, at least heuristically, the error $u(F)-U_\ell(F)$ locally on 
each element $T\in\TT_\ell$.

\begin{algorithm}\label{algorithm}
{\sc Input:} Right-hand side $F\in\HH^*$, initial mesh $\TT_0$ with $\XX_0\supseteq \XX_{00}$, and bulk
parameter $0<\theta\le1$.\\ 
For $\ell=0,1,2,\dots$ iterate the following:
\begin{itemize}
\item[\rm(i)] Compute Galerkin solution $U_\ell(F)\in\XX_\ell$.
\item[\rm(ii)] Compute refinement indicators $\mu_\ell(F;T)$ for all $T\in\TT_\ell$.
\item[\rm(iii)] Determine some set $\MM_\ell\subseteq\TT_\ell$ of marked elements which satisfies
\begin{align}\label{eq:doerfler}
 \theta\,\mu_\ell(F)^2 \le \mu_\ell(F;\MM_\ell)^2.
\end{align}
\item[\rm(iv)] Generate a new mesh $\TT_{\ell+1}$ and hence an enriched space 
$\XX_{\ell+1}$ by refinement of at least all marked elements $T\in\MM_\ell$.
\end{itemize}
{\sc Output:} Sequence of successively refined triangulations $\TT_\ell$ as well as 
corresponding Galerkin solutions $U_\ell(F)\in\XX_\ell$ and error estimators $\mu_\ell(F)$, 
for $\ell\in\N_0$.
\end{algorithm}

\subsection{Auxiliary estimator and assumptions}
The following convergence results of Proposition~\ref{prop:convergence}
and Proposition~\ref{prop:new} 
require an auxiliary error estimator
\begin{align}
 \rho_\ell(F) := \rho_\ell(F;\TT_\ell)
 \quad\text{with}\quad
 \rho_\ell(F;\EE_\ell):=\Big(\sum_{T\in\EE_\ell}\rho_\ell(F;T)^2\Big)^{1/2}<\infty
 \quad\text{for all }\EE_\ell\subseteq\TT_\ell
\end{align}
with local contributions $\rho_\ell(F;T)\ge0$. 
For all $\ell\in\N_0$, we suppose that there exists some superset 
$\MM_\ell\subseteq\RR_\ell\subseteq\TT_\ell$ which satisfies the 
following three assumptions~\eqref{ass:local}--\eqref{ass:stable}:
\begin{enumerate}
\renewcommand{\theenumi}{A\arabic{enumi}}
\item\label{ass:local}
$\mu_\ell(F)$ is a local lower bound of $\rho_\ell(F)$: There is a constant
$\setc{local}>0$ such that for all $\ell\in\N_0$ holds
\begin{align}
 \mu_\ell(F;\MM_\ell)\le\c{local}\,\rho_\ell(F;\RR_\ell).
\end{align}
\item\label{ass:contraction}
$\rho_\ell(F)$ is contractive on $\RR_\ell$: There is a constant $\setc{inv}>0$ 
such that for all $\ell,k\in\N_0$ and all $\delta>0$ holds
\begin{align}
 \c{inv}^{-1}\,\rho_\ell(F;\RR_\ell)^2
 \le \rho_\ell(F)^2 - \frac{1}{1+\delta}\,\rho_{\ell+k}(F)^2
 + (1+\delta^{-1})\c{inv}\,\norm{U_{\ell+k}(F)-U_\ell(F)}\HH^2.
\end{align}
\end{enumerate}
The constants $\c{local},\c{inv}>0$ may depend on $F$, but are independent of the step $\ell\in\N_0$, i.e., in particular independent of 
the discrete spaces $\XX_\ell$ and the corresponding Galerkin solutions 
$U_\ell(F)$. If $\rho_\ell(F)$ is not well-defined for all $F\in\HH^*$, but
only on a dense subset $D\subseteq\HH^*$, we require the following additional
assumption:
\begin{enumerate}
\renewcommand{\theenumi}{A\arabic{enumi}}
\setcounter{enumi}{2}
\item\label{ass:stable}
$\mu_\ell(\cdot)$ is stable on $\MM_\ell$ with respect to $F$:
There is a constant $\setc{stable}>0$ such that for all $\ell\in\N_0$ and 
$F'\in\HH^*$ 
holds
\begin{align}
|\mu_\ell(F;\MM_\ell)-\mu_\ell(F';\MM_\ell)|
 \le \c{stable}\norm{F-F'}{\HH^*}.
\end{align}
\end{enumerate}

\definecolor{gray}{rgb}{.5,.5,.5}
\def\next{{\color{gray}\Large$\bullet$~}}
\subsection{Remarks}
Some remarks are in order to relate the abstract 
assumptions~\eqref{ass:local}--\eqref{ass:stable} to the applications,
we have in mind.

\next
\textbf{Choice of $\boldsymbol{\rho_\ell}$.}
Below, we shall verify that assumptions~\eqref{ass:local}--\eqref{ass:stable}
hold with 
$\mu_\ell(F)$ being the Faermann error estimator~\cite{faermann2d,faermann3d,cf} 
for BEM resp.\ $\mu_\ell(F)$ being the two-level error estimator for 
BEM~\cite{msw,mms,ms,hms,heuer,eh,effp,efgp,hypsing3d} and the FEM-BEM coupling~\cite{ms:fembem,gms,afkp:fembem}. In either case, $\rho_\ell(F)$ denotes some
weighted-residual error estimator, see~\cite{cs95,cs96,cc97,cms,cmps} for BEM
and~\cite{cs95:fembem,gms,affkmp:fembem} for the FEM-BEM coupling. 

\next
\textbf{Necessity of~(\ref{ass:stable}).}
In these cases, the weighted-residual error estimator $\rho_\ell$ imposes additional regularity assumptions on the 
given right-hand side $F$. For instance, the weighted-residual error estimator
for the weakly-singular integral equation~\cite{cs95,cs96,cc97,cms} 
requires $F\in H^1(\Gamma)$, while the natural space for the residual is 
$H^{1/2}(\Gamma)$, see Section~\ref{section:symm} for further details and 
discussions. Convergence~\eqref{eq:convergence} of Algorithm~\ref{algorithm} for arbitrary $F\in H^{1/2}(\Gamma)$ then follows
by means of stability~\eqref{ass:stable}.

\next
\textbf{Verification of~(\ref{ass:local})--(\ref{ass:contraction}).}
For two-level estimators, \eqref{ass:local} has first been observed
in~\cite{cf,cmps} for BEM and~\cite{afkp:fembem} for the FEM-BEM coupling and 
follows essentially from scaling arguments for the hierarchical basis functions. 
For the Faermann error estimator and a simplified 2D BEM setting,
\eqref{ass:local} is also proved in~\cite{cf}.
Finally, the novel observation~\eqref{ass:contraction} follows from an 
appropriately constructed mesh-size function and refinement of marked elements
as well as appropriate inverse-type 
estimates, where we shall build on the recent developments of~\cite{afembem};
see e.g.\ the proof of Theorem~\ref{thm:symm:twolevel}.

\next
\textbf{Verification of~(\ref{ass:stable}).}
Suppose that the operator $A$ is linear and $\mu_\ell(\cdot)$ is efficient
\begin{align}\label{eq:efficient} 
 \mu_\ell(F) \le \Ceff\,\norm{u(F)-U_\ell(F)}\HH
 \quad\text{for all }F\in\HH^*.
\end{align}
Provided $\mu_\ell(\cdot)$ has a semi-norm structure, the
corresponding triangle inequality yields
\begin{align}
\begin{split}\label{eq:18}
 \mu_\ell(F)
 \le \mu_\ell(F') + \mu_\ell(F-F')
 &\le \mu_\ell(F') + \Ceff\,\norm{u(F-F') - U_\ell(F-F')}\HH
 \\&
 \le \mu_\ell(F') + \Ceff\c{cea}\,\norm{u(F-F')}\HH
 \\&
 \le \mu_\ell(F') + \Ceff\c{cea}\,\norm{A^{-1}}{}\,\norm{F-F'}{\HH^*},
\end{split}
\end{align}
where $\norm{A^{-1}}{}$ denotes the operator norm of $A^{-1}$, and the
(bounded) inverse exists due to~\eqref{eq:cont}. This proves stability~\eqref{ass:stable}
with $\c{stable} = \Ceff\c{cea}\,\norm{A^{-1}}{}$.

\next
\textbf{Marking strategy.}
In view of optimal convergence rates, one usually asks for 
$\#\RR_\ell\lesssim\#\MM_\ell$ in~\eqref{ass:local}
and minimal cardinality of $\MM_\ell$ in~\eqref{eq:doerfler}. We stress, however, that 
this is not necessary for the present analysis, where our focus is on a
first plain convergence result.

\subsection{Abstract convergence analysis}
We start with the observation that~\eqref{ass:contraction}
already implies convergence of the auxiliary estimator $\rho_\ell$. We note
that the following lemma is, in particular, independent of the marking
strategy~\eqref{eq:doerfler}, i.e., we do not use any information about 
how the sequence $(\TT_\ell)_{\ell\in\N_0}$ is generated.

\begin{lemma}\label{prop:aux-convergence}
Suppose \eqref{ass:contraction} for some
fixed $F\in\HH^*$. Under nestedness $\XX_\ell\subseteq\XX_{\ell+1}$ of the 
discrete spaces for all $\ell\in\N_0$, the auxiliary estimator $\rho_\ell(F)$ 
converges, i.e, the limit
\begin{align}\label{gme:conv1}
 \rho_\infty(F) := \lim_{\ell\to\infty}\rho_\ell(F)
\end{align}
exists in $\R$. Moreover, it holds
\begin{align}\label{gme:conv2}
 \lim_{\ell\to\infty}\rho_\ell(F;\RR_\ell)=0.
\end{align}
\end{lemma}

\begin{proof}
First, we prove that~\eqref{ass:contraction} implies boundedness of 
$(\rho_\ell)_{\ell\in\N_0}$. 
We recall that nestedness $\XX_{\ell}\subseteq\XX_{\ell+1}$ for all $\ell\in\N_0$
in combination with the C\'ea lemma~\eqref{eq:cea} implies that the
limit $\lim_\ell U_\ell(F) =: U_\infty(F)$ exists in $\HH$, see 
e.g.~\cite{msv,cp,afp} or even the pioneering work~\cite{bv}. 
For $\ell=0$ and $\delta=1$, assumption~\eqref{ass:contraction} implies
\begin{align*}
 \frac12\,\rho_k(F)^2 \le \rho_0(F)^2 + 2\c{inv}\,\sup_{k\in\N_0}
 \norm{U_0-U_k}{\HH}^2 \le M < \infty.
\end{align*}
Next, we multiply ~\eqref{ass:contraction} by $(1+\delta)$ and observe
\begin{align}\label{eq:contraction2}
 \hspace*{-2mm}%
 0 \le 
 \rho_\ell(F;\RR_\ell)^2 
 \lesssim \rho_\ell(F)^2 - \rho_{\ell+k}(F)^2 + \delta \rho_\ell(F)^2
 + \c{inv}(\delta) \norm{U_{\ell+k}(F)\!-\!U_\ell(F)}\HH^2
\end{align}
with $\c{inv}(\delta):=(1+\delta)(1+\delta^{-1})\c{inv}=\delta^{-1}(1+\delta)^2\c{inv}$.
Let $\eps>0$. Because of the boundedness of $\rho_\ell(F)$, we can hence 
choose  $\delta>0$ and $\ell_0\in\N$ such that
$$
\delta \rho_\ell(F)^2 + \c{inv}(\delta)\,\norm{U_{\ell+k}(F)-U_\ell(F)}\HH^2 \le \eps
$$
for all $\ell\ge\ell_0$ and $k\in\N_0$.
Together with~\eqref{eq:contraction2}, this shows
\begin{align}\label{eq:help:cont}
\rho_\ell(F)^2 - \rho_{\ell+k}(F)^2 \ge -\eps.
\end{align}
Let $a,b\in\R$ be accumulation points of $(\rho_\ell(F)^2)_{\ell\in\N_0}$. First, choose $\ell\ge\ell_0$ and $k\in\N$ such that
$|\rho_\ell(F)^2-a|+|\rho_{\ell+k}(F)^2-b|\leq\eps$. With~\eqref{eq:help:cont}, this implies
\begin{align*}
 a-b\geq -3\eps.
\end{align*}
Second, choose $\ell\ge\ell_0$ and $k\in\N$ such that
$|\rho_\ell(F)^2-b|+|\rho_{\ell+k}(F)^2-a|\leq\eps$ to derive
\begin{align*}
 b-a\geq -3\eps.
\end{align*}
Since $\eps>0$ was arbitrary, the last two estimates imply $a=b$. Altogether, $(\rho_\ell(F)^2)_{\ell\in\N_0}$ is a bounded 
sequence in $\R$ with unique accumulation point. By elementary calculus,
$(\rho_\ell(F)^2)_{\ell\in\N_0}$ is convergent with limit $\rho_\infty(F)^2$.
Continuity of the square root concludes~\eqref{gme:conv1}.
In particular, this and~\eqref{eq:contraction2} prove 
$\rho_\ell(F;\RR_\ell)\to 0$ as $\ell\to\infty$.
\end{proof}

\begin{proposition}\label{prop:convergence}
Suppose assumptions~\eqref{ass:local}--\eqref{ass:contraction} for some
fixed $F\in\HH^*$. Under nestedness $\XX_\ell\subseteq\XX_{\ell+1}$ of the 
discrete spaces for all $\ell\in\N_0$ and due to the marking 
strategy~\eqref{eq:doerfler}, Algorithm~\ref{algorithm} guarantees estimator 
convergence $\displaystyle\lim_{\ell\to\infty}\mu_\ell(F)=0$.
\end{proposition}

\begin{proof}
The marking criterion~\eqref{eq:doerfler} and assumption~\eqref{ass:local} show
$$\theta\mu_{\ell}(F)^2 \le \mu_{\ell}(F;\MM_{\ell})^2 \lesssim \rho_{\ell}(F;\RR_{\ell})^2.$$
Hence, the assertion $\mu_{\ell}(F)\xrightarrow{\ell\to\infty} 0$ follows from 
Lemma~\ref{prop:aux-convergence}.
\end{proof}

\begin{proposition}\label{prop:new}
Suppose that $D\subseteq\HH^*$ is a dense subset of $\HH^*$ such that 
assumptions~\eqref{ass:local}--\eqref{ass:contraction} are satisfied for all
$F\in D$. In addition, suppose validity of~\eqref{ass:stable}.
Under nestedness $\XX_\ell\subseteq\XX_{\ell+1}$ of the 
discrete spaces for all $\ell\in\N_0$ and due to the marking 
strategy~\eqref{eq:doerfler}, Algorithm~\ref{algorithm} guarantees  
convergence $\displaystyle\lim_{\ell\to\infty}\mu_\ell(F)=0$ for all $F\in\HH^*$.
\end{proposition}

\begin{proof}
Let $\eps>0$ and choose $F^\prime\in D$ such that $\norm{F-F^\prime}{\HH^*}\leq \eps$. The marking criterion~\eqref{eq:doerfler} 
as well as~\eqref{ass:stable} and~\eqref{ass:local} show
$$ \theta\,\mu_\ell(F) \le \mu_\ell(F;\MM_\ell) \lesssim \mu_\ell(F';\MM_\ell) + \norm{F-F'}{\HH^*} \lesssim \rho_\ell(F';\RR_\ell) + \eps.$$
Lemma~\ref{prop:aux-convergence} yields $\rho_\ell(F';\RR_\ell)\xrightarrow{\ell\to\infty} 0$, whence
\begin{align*}
 \theta\,\limsup_{\ell\to\infty} \mu_\ell(F)\lesssim \eps.
\end{align*}
With $\eps\to 0 $, elementary calculus concludes the proof.
\end{proof}

\def\supp{{\rm supp}}
\section{Weakly-singular integral equation}
\label{section:symm}

\subsection{Model problem}
We consider the weakly-singular integral equation
\begin{align}\label{eq:symm}
 Au(x) = \int_\Gamma G(x-y)\,u(y)\,d\Gamma(y) = F(x)
 \quad\text{for all }x\in\Gamma
\end{align}
on a relatively open, polygonal part $\Gamma\subseteq\partial\Omega$ of the boundary of a 
bounded, polyhedral Lipschitz domain $\Omega\subset\R^d$, $d=2,3$. For $d=3$, we assume that the boundary of $\Gamma$ (a polygonal curve) is Lipschitz itself.
Here, 
\begin{align}\label{eq:G}
 G(z) = -\frac{1}{2\pi}\log|z|
 \quad\text{resp.}\quad
 G(z) = \frac{1}{4\pi}|z|^{-1}
\end{align} 
denotes the fundamental solution of the Laplacian in $d=2,3$. 
The reader is referred to, e.g., the monographs~\cite{hw,mclean,ss,steinbach} for
proofs of and details on the following facts:
The simple-layer integral operator
$A:\HH\to\HH^*$ is a continuous linear operator between the fractional-order Sobolev
space $\HH=\H^{-1/2}(\Gamma)$ and its dual $\HH^*=H^{1/2}(\Gamma):=\set{\widehat v|_\Gamma}{\widehat v\in H^1(\Omega)}$. Duality is 
understood with respect to the extended $L^2(\Gamma)$-scalar product $\dual\cdot\cdot$. 
In 2D, we additionally assume
$\diam(\Omega)<1$ which can always be achieved by scaling. Then, the simple-layer
integral operator is also elliptic
\begin{align}
 \dual{v}{Av} \ge \c{elliptic}\,\norm{v}{\H^{-1/2}(\Gamma)}^2
 \quad\text{for all }v\in\HH=\H^{-1/2}(\Gamma)
\end{align}
with some constant $\setc{elliptic}>0$ which depends only on $\Gamma$.
Thus, $A$ meets all assumptions of Section~\ref{section:abstract}, 
and $\norm{v}A^2:=\dual{Av}{v}$ even defines an equivalent Hilbert norm
on $\HH$.

\subsection{Discretization}
\label{section:discretization}
Let $\TT_\star$ be a $\gamma$-shape regular triangulation of $\Gamma$
into  affine line segments for $d=2$ resp.\ plane surface triangles for $d=3$.
For $d=3$, $\gamma$-shape regularity means
\begin{subequations}\label{eq:shaperegular}
\begin{align}
 \sup_{T\in\TT_\star}\frac{\diam(T)^2}{|T|} \le \gamma < \infty
\end{align}
with $|\cdot|$ being the two-dimensional surface measure, whereas
for $d=2$, we impose uniform boundedness of the local mesh-ratio
\begin{align}
 \frac{\diam(T)}{\diam(T')} \le \gamma < \infty
 \quad\text{for all }T,T'\in\TT_\star\text{ with }T\cap T'\neq\emptyset.
\end{align}
\end{subequations}
To abbreviate notation, we shall write $|T|:=\diam(T)$ for $d=2$.
In addition, we assume that $\TT_\star$ is regular in the sense of Ciarlet
for $d=3$, i.e., there are no hanging nodes.

With $\XX_\star=\PP^0(\TT_\star)$ being the space of $\TT_\star$-piecewise
constant functions, we now consider the Galerkin formulation~\eqref{eq:galerkin}.

\subsection{Weighted-residual error estimator}
According to the Galerkin formulation~\eqref{eq:galerkin}, the
residual $F-AU_\star(F)\in H^{1/2}(\Gamma)$ has $\TT_\star$-piecewise
integral mean zero, i.e.,
\begin{align}\label{eq:symm:mean:zero}
 \int_T (F-AU_\star(F))\,d\Gamma = 0
 \quad\text{for all }T\in\TT_\star.
\end{align}
Suppose for the moment that the right-hand side has additional regularity
$F\in H^1(\Gamma)\subset H^{1/2}(\Gamma)$.
Since $A:\H^{-1/2}(\Gamma)\to H^{1/2}(\Gamma)$ is an isomorphism with additional stability $A:\H^{-1/2+s}(\Gamma)\to H^{1/2+s}(\Gamma)$ for all $-1/2\le s\le 1/2$ (We note that $A$ is \emph{not} isomorphic for $s=\pm1$ and $\Gamma\subsetneqq\partial\Omega$.),
a Poincar\'e-type inequality in $H^{1/2}(\Gamma)$ shows
\begin{align}\label{eq:resmean}
\begin{split}
 \norm{u(F)-U_\star(F)}{\H^{-1/2}(\Gamma)}
 \simeq \norm{F-AU_\star(F)}{H^{1/2}(\Gamma)}
 &\lesssim \norm{h_\star^{1/2}\nabla_\Gamma(F-AU_\star(F))}{L^2(\Gamma)}
 \\&=:\eta_\star(F),
\end{split}
\end{align}
see~\cite{cs95,cs96,cc97,cms}.
Here, $\nabla_\Gamma(\cdot)$ denotes the surface gradient, and
$h_\star\in\PP^0(\TT_\star)$ is the local mesh-width function defined 
pointwise almost everywhere by $h_\star|_T := \diam(T)$ for all $T\in\TT_\star$.
Overall, this proves the reliability estimate
\begin{align}
 \norm{u(F)-U_\star(F)}{\H^{-1/2}(\Gamma)}
 \le \CCrel\,\eta_\star(F),
\end{align}
and the constant $\CCrel>0$ depends only on $\Gamma$
and the $\gamma$-shape regularity~\eqref{eq:shaperegular} of $\TT_\star$; see~\cite{cms}. In 2D, it holds that
$\CCrel = C\,\log^{1/2}(1+\gamma)$, where $C>0$ depends only on $\Gamma$; see~\cite{cc97}.
In particular, the weighted-residual error estimator can be localized via
\begin{align}\label{eq:symm:resest}
 \eta_\star(F) = \Big(\sum_{T\in\TT_\star}\eta_\star(F;T)^2\Big)^{1/2}
 \text{ with }
 \eta_\star(F;T) = \diam(T)^{1/2}\norm{\nabla_\Gamma(F-AU_\star(F))}{L^2(T)}.
\end{align} 
Recently, convergence of Algorithm~\ref{algorithm} has been shown 
even with quasi-optimal rates, if $\eta_\ell(F)=\mu_\ell(F)$ is used for 
marking~\eqref{eq:doerfler}; see~\cite{abem,ffkmp:part1}. We stress that our approach
with $\eta_\ell(F) = \rho_\ell(F) = \mu_\ell(F)$ would also give convergence 
$\eta_\ell(F) \to 0$ as $\ell\to\infty$. Since this is, however, a much weaker
result than that of~\cite{abem}, we omit the details.

\begin{figure}[t]
\psfrag{+1}[c][c]{\footnotesize$+1$}
\psfrag{-1}[c][c]{\footnotesize$-1$}
\def\fig#1#2{\begin{minipage}{3cm}\centering\includegraphics[width=\textwidth]{#1.eps}\\#2\end{minipage}}
\fig{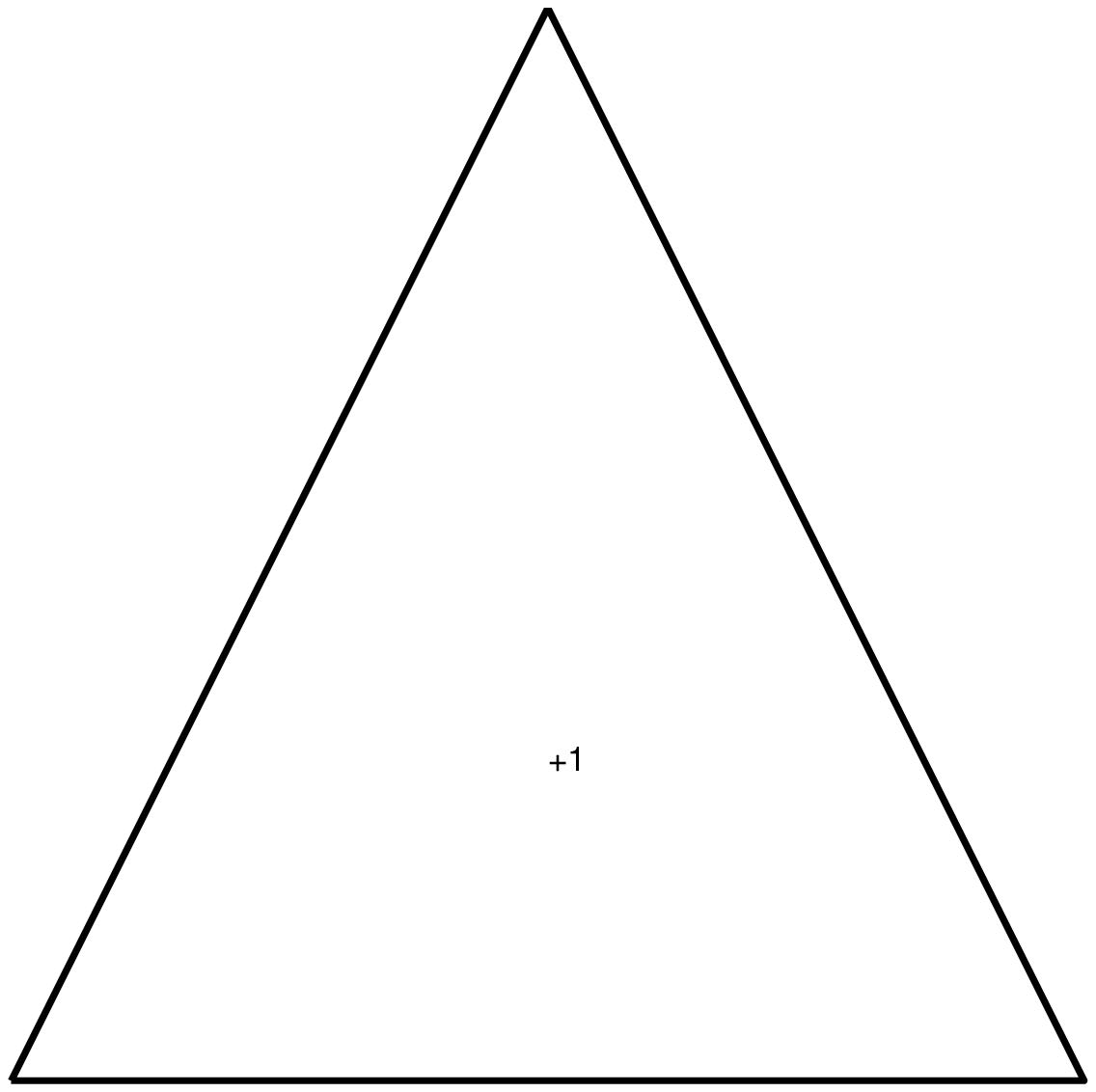}{$\chi_T$}\quad
\fig{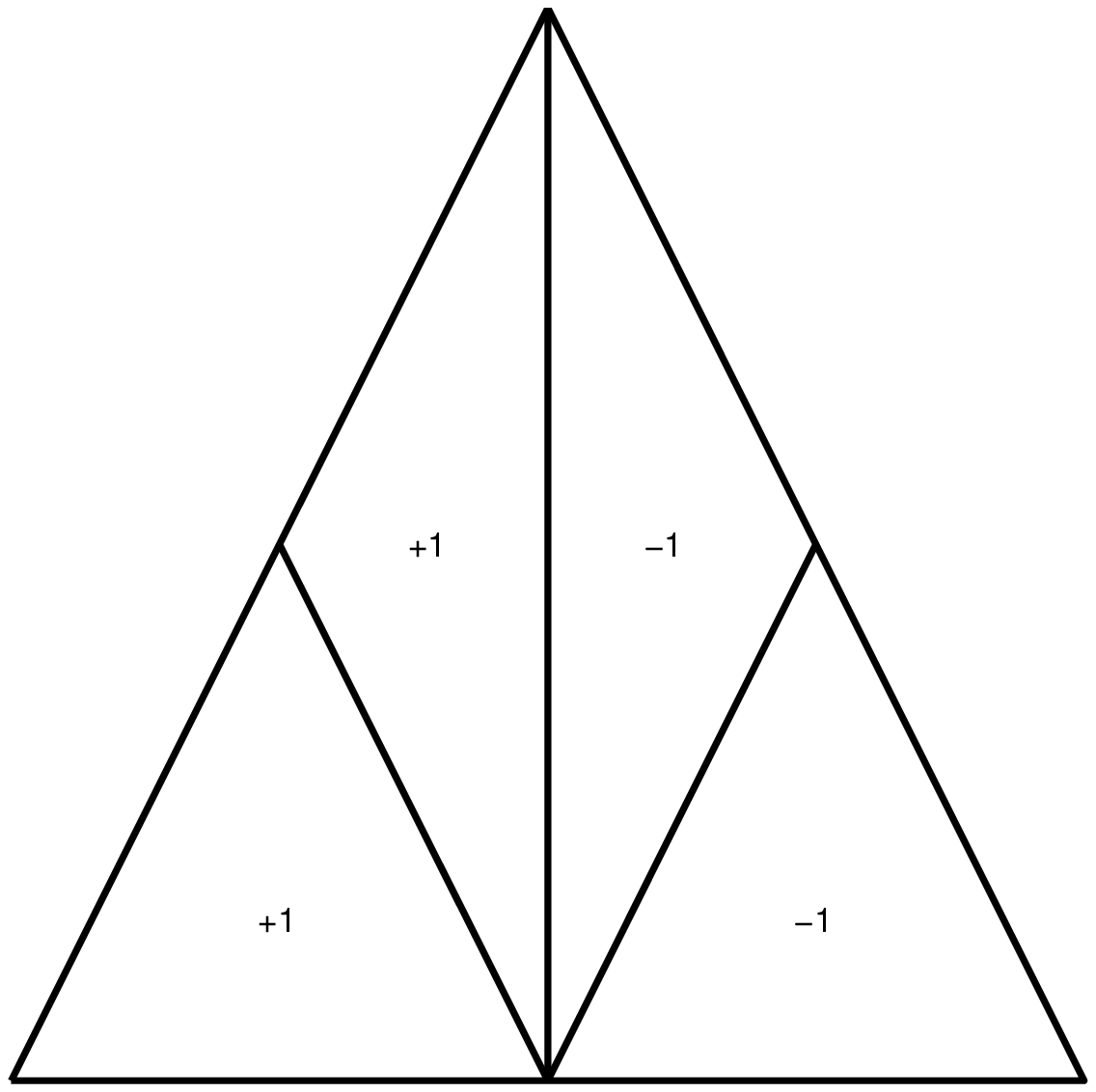}{$\varphi_{T,1}$}\quad
\fig{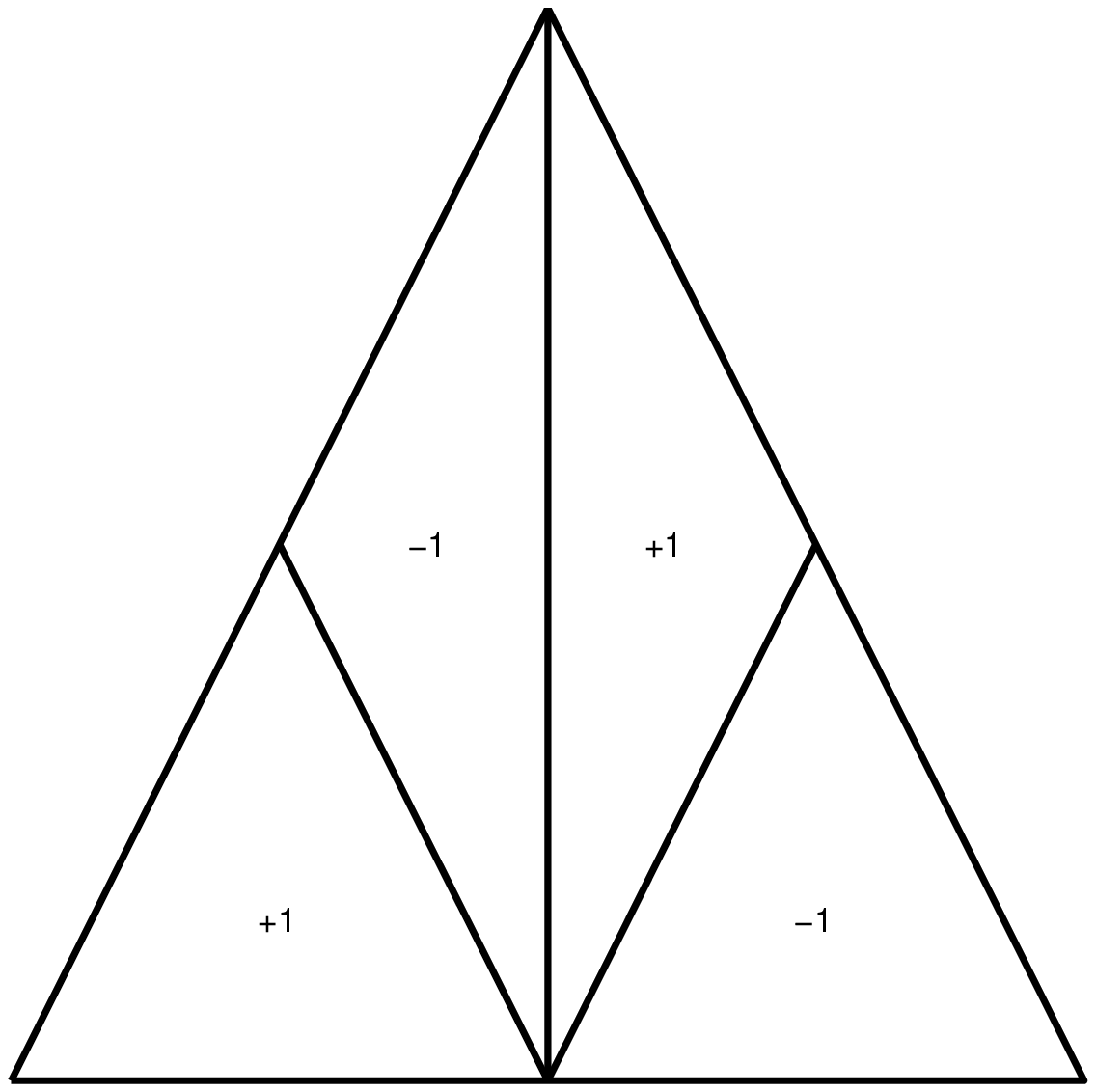}{$\varphi_{T,2}$}\quad
\fig{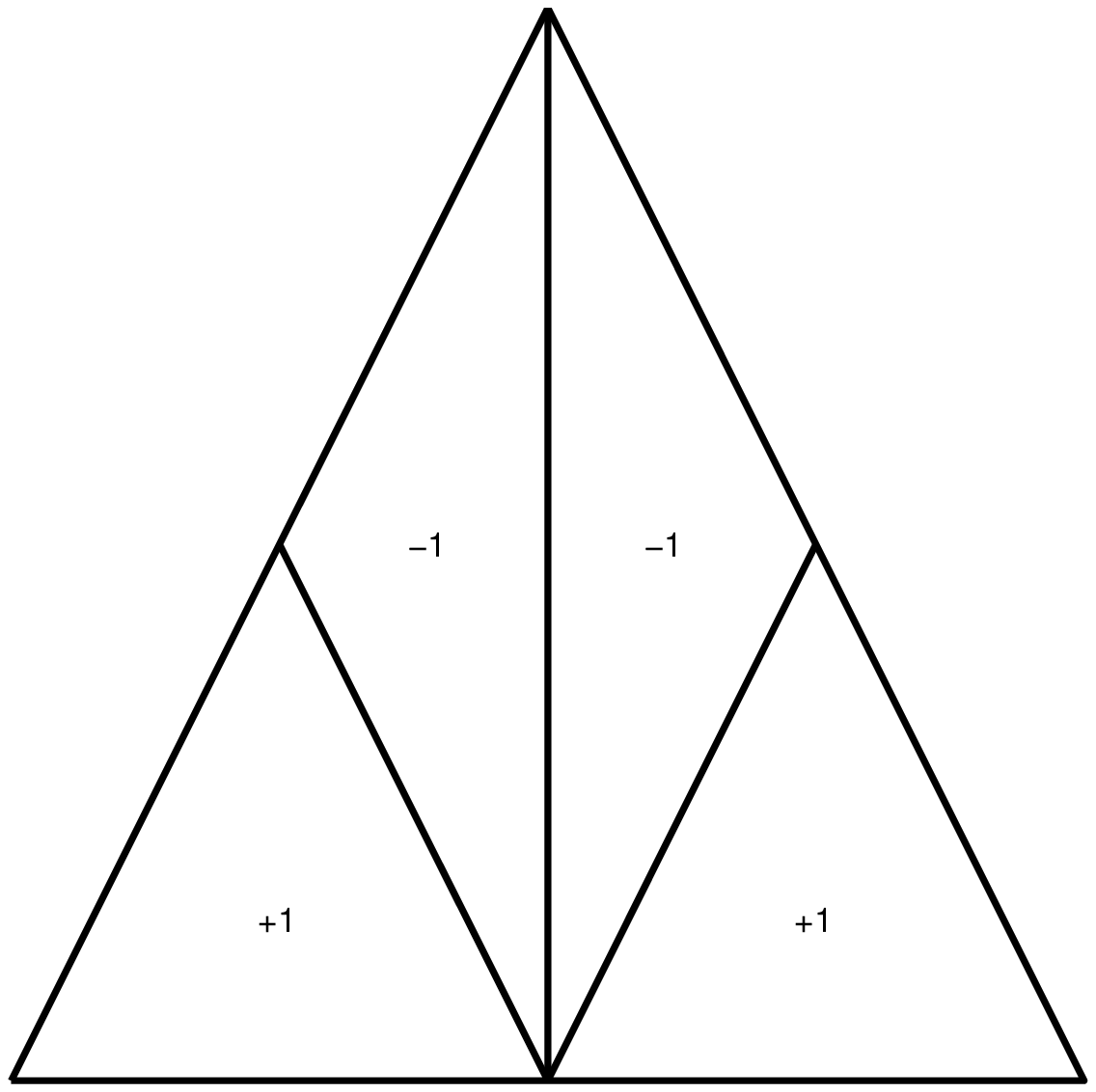}{$\varphi_{T,3}$}\quad
\caption{For $d=3$, uniform bisection-based mesh-refinement usually splits
a coarse mesh element $T\in\TT_\ell$ (left) into four sons $T'\in\widehat\TT_{\ell}$
(right) so that $|T|/4 = |T'|$. 
Typical hierarchical basis functions $\varphi_{T,j}$ are indicated by their piecewise constant 
values $\pm1$ on the son elements $T'$.}
\label{fig:twolevel:3D}
\end{figure}
\subsection{Two-level error estimator}
In the frame of weakly-singular integral equations~\eqref{eq:symm}, the
two-level error estimator was introduced in~\cite{msw}.
Let $\widehat\TT_\star$ denote the uniform refinement of $\TT_\star$.
For each element $T\in\TT_\star$, let 
$\widehat\TT_\star|_T:=\set{T'\in\widehat\TT_\star}{T'\subset T}$ denote the 
set of sons of $T$. Let $\{\chi_T,\varphi_{T,1},\dots,\varphi_{T,D}\}$
be a basis of $\PP^0(\widehat\TT_\star|_T)$ with fine-mesh functions 
$\varphi_{T,j}$ which satisfy $\supp(\varphi_{T,j})\subseteq T$ and
$\int_T\varphi_{T,j}\,d\Gamma=0$. We note 
that usually $D=1$ for $d=2$ and $D=3$ for $d=3$. Typical choices are shown in 
Figure~\ref{fig:twolevel:3D}. Then, the 
local contributions of the two-level error estimator from~\cite{msw,mms,hms,eh,effp} read
\begin{align}\label{eq:symm:twolevel}
\mu_\star(F;T)^2 = \sum_{j=1}^D\mu_{\star,j}(F;T)^2
\quad\text{with}\quad
\mu_{\star,j}(F;T) = \frac{\dual{F-AU_\star(F)}{\varphi_{T,j}}}{\dual{A\varphi_{T,j}}{\varphi_{T,j}}^{1/2}}.
\end{align}
Put differently, we test the residual $F-AU_\star(F)\in H^{1/2}(\Gamma)$ 
with the additional basis functions from 
$\PP^0(\widehat\TT_\star)\backslash\PP^0(\TT_\star)$.
This quantity is appropriately scaled by the corresponding energy norm 
$\norm{\varphi}{\H^{-1/2}(\Gamma)}\simeq\dual{A\varphi}{\varphi}^{1/2}=\norm{\varphi}A$.
Note that unlike the weighted-residual error estimator $\eta_\star(\cdot)$
from~\eqref{eq:symm:resest},  the two-level
error estimator $\mu_\star(F)$ is well-defined under minimal regularity 
$F\in H^{1/2}(\Gamma)$ of the given right-hand side.

The two-level estimator is known to be efficient~\cite{msw,mms,hms,eh,effp}
\begin{align}\label{eq:twolevel:efficient}
 \mu_\star(F) \le \Ceff\,\norm{u(F)-U_\star(F)}{\H^{-1/2}(\Gamma)},
\end{align}
while reliability
\begin{align}\label{eq:twolevel:reliable}
 \norm{u(F)-U_\star(F)}{\H^{-1/2}(\Gamma)} \le \Crel\,\mu_\star(F) 
\end{align}
holds under~\cite{msw,mms,hms,eh} and is even equivalent to~\cite{effp} the saturation assumption
\begin{align}
\norm{u(F)-\widehat U_\star(F)}A \le q_{\rm sat}\,\norm{u(F)-U_\star(F)}A
\end{align}
in the energy norm $\norm\cdot{A}\simeq\norm\cdot{\H^{-1/2}(\Gamma)}$.
Here, $0<q_{\rm sat}<1$ is a uniform constant, and $\widehat U_\star(F)$ is the Galerkin
solution with respect to the uniform refinement $\widehat\TT_\star$ of $\TT_\star$.
The constant $\Ceff>0$ depends only on $\Gamma$ and $\gamma$-shape regularity of
$\TT_\star$, while $\Crel>0$ additionally depends on the saturation constant 
$q_{\rm sat}$.

With the help of Proposition~\ref{prop:convergence} and Proposition~\ref{prop:new}, we aim to prove the
following convergence result for the related adaptive mesh-refining algorithm.
Recall that for $d=3$, refinement of an element $T\in\TT_\ell$ does not 
necessarily imply that $\diam(T') < \diam(T)$ for the sons $T'\in\TT_{\ell+1}$
of $T$. However, it is reasonable to assume that each marked element 
$T\in\MM_\ell$ is refined into at least two sons $T'\in\TT_{\ell+1}$ which 
satisfy $|T'|\le \kappa\,|T|$ with some uniform $0<\kappa<1$ (and $\kappa=1/2$
for usual mesh-refinement strategies for $d=2,3$).

\begin{theorem}\label{thm:symm:twolevel}
Suppose that the two-level error estimator~\eqref{eq:symm:twolevel} is used
for marking~\eqref{eq:doerfler}. Suppose that the mesh-refinement
guarantees uniform $\gamma$-shape regularity~\eqref{eq:shaperegular} of the\linebreak meshes $\TT_\ell$ generated,
as well as that all marked elements $T\in\MM_\ell$ are refined into sons 
$T'\in\TT_{\ell+1}$ with $|T'|\le \kappa\,|T|$ with some uniform constant $0<\kappa<1$.
Then, Algorithm~\ref{algorithm} guarantees 
\begin{align}\label{symm:convergence}
 \mu_\ell(F)\to0 
 \quad\text{as }\ell\to\infty
\end{align}
for all $F\in H^{1/2}(\Gamma)$.
\end{theorem}

The claim of Theorem~\ref{thm:symm:twolevel} 
follows from 
Proposition~\ref{prop:new}
as soon as we have verified the abstract assumptions~\eqref{ass:local}--\eqref{ass:stable}.
We will show~\eqref{ass:local}--\eqref{ass:contraction} for a slight variant $\rho_\star(\cdot)$
of the weighted-residual error estimator $\eta_\star(\cdot)$ from~\eqref{eq:symm:resest} and for all right-hand sides 
$F\in H^1(\Gamma)$. Afterward, assumption
\eqref{ass:stable} is shown for all $F\in H^{1/2}(\Gamma)$, and the
final claim then follows from density of $H^1(\Gamma)$ within $H^{1/2}(\Gamma)$.

\begin{proof}[Proof of 
Theorem~\ref{thm:symm:twolevel}]
For given right-hand side $F\in H^1(\Gamma)$, the weighted-residual error estimator 
$\eta_\star(F)$ from~\eqref{eq:symm:resest} is well-defined.

\next 
Note that $\gamma$-shape regularity~\eqref{eq:shaperegular} implies for $d=3$ the pointwise equivalence
\begin{align}
 \c{meshsize}^{-1}\diam(T)
 \le |T|^{1/2} \le \diam(T)
 \quad\text{for all }T\in\TT_\ell,
\end{align}
where $\setc{meshsize}=\sqrt{\gamma}>0$. In the spirit of~\cite{ckns},
we hence use the modified mesh-width function 
$\widetilde h_\ell\in\PP^0(\TT_\ell)$ defined pointwise almost everywhere
by $\widetilde h_\ell|_T = |T|^{1/(d-1)}$
and note that $\widetilde h_\ell = h_\ell$ for $d=2$. 
Then, we consider an equivalent weighted-residual error estimator $\rho_\ell(F)$ given by
\begin{align}\label{eq:symm:equivalence}
 \c{meshsize}^{-1/2}\,\eta_\ell(F;T)
 \le \rho_\ell(F;T) 
 := \norm{\widetilde h_\ell^{1/2}\nabla_\Gamma(F-AU_\ell(F))}{L^2(T)}
 \le \eta_\ell(F;T)
\end{align}

\next 
It has first been noted in~\cite[Theorem~8.1]{cf} for 2D that
\begin{align}\label{eq:symm:local}
 \mu_{\ell,j}(F;T) 
 \le\c{locall} \eta_\ell(F;T)
 \quad\text{for all }T\in\TT_\ell,
\end{align}
where the constant $\setc{locall}>0$ depends only on $\gamma$-shape regularity
of $\TT_\ell$, and the proof transfers to 3D as well. For completeness, we 
include the short argument: With $\supp(\varphi_{T,j})\subseteq T$, we infer
\begin{align}\label{du1}
 \mu_{\ell,j}(F;T) = \frac{\dual{F-AU_\ell(F)}{\varphi_{T,j}}}{\norm{\varphi_{T,j}}{A}}
 \le \norm{h_\ell^{-1/2}(F-AU_\ell(F))}{L^2(T)}
 \, \frac{\norm{h_\ell^{1/2}\varphi_{T,j}}{L^2(T)}}{\norm{\varphi_{T,j}}{A}}.
\end{align}
With the inverse estimate from~\cite[Theorem~3.6]{ghs} and norm equivalence, we obtain
\begin{align*}
 \norm{h_\ell^{1/2}\varphi_{T,j}}{L^2(T)}
 = \norm{h_\ell^{1/2}\varphi_{T,j}}{L^2(\Gamma)}
 \lesssim \norm{\varphi_{T,j}}{\H^{-1/2}(\Gamma)}
 \simeq \norm{\varphi_{T,j}}{A},
\end{align*}
where the hidden constants depend only on $\Gamma$ and $\gamma$-shape regularity~\eqref{eq:shaperegular} of $\TT_\ell$. We note that the assumption $\int_T\varphi_{T,j}\,d\Gamma=0$
together with the approximation result of~\cite[Theorem~4.1]{ccdpr:symm} also proves
the converse estimate
\begin{align*}
 \norm{\varphi_{T,j}}{A}
 \simeq\norm{\varphi_{T,j}}{\H^{-1/2}(\Gamma)}
 \lesssim \norm{h_\ell^{1/2}\varphi_{T,j}}{L^2(T)},
\end{align*}
where the hidden constant depends only on $\Gamma$. This proves that the quotient
on the right-hand side of~\eqref{du1} remains bounded. Due to~\eqref{eq:symm:mean:zero}, the Poincar\'e estimate yields $\norm{h_\ell^{-1/2}(F-AU_\ell(F))}{L^2(T)}\lesssim \norm{h_\ell^{1/2}\nabla_\Gamma(F-AU_\ell(F)}{L^2(T)}$. This concludes~\eqref{eq:symm:local}.
Together with~\eqref{eq:symm:equivalence}, this 
proves~\eqref{ass:local} with $\c{local}=\c{locall}\c{meshsize}^{1/2}D^{1/2}$
and $\RR_\ell = \MM_\ell$.


\next 
The verification of~\eqref{ass:contraction} hinges on the use of the 
equivalent mesh-size function. Note that each marked element $T\in\MM_\ell = \RR_\ell$ 
is refined and that the mesh-size sequence is pointwise decreasing. With 
$q=\kappa^{1/(d-1)}$, this implies the pointwise estimate
\begin{align*}
 \widetilde h_\ell - \widetilde h_{\ell+k}
 \ge \widetilde h_\ell - \widetilde h_{\ell+1}
 \ge (1-q)\, \widetilde h_\ell \chi_{\bigcup\RR_\ell}
 \quad\text{for all }\ell,k\in\N,
\end{align*}
where $\chi_{\bigcup\RR_\ell}$ denotes the characteristic function of the
set $\bigcup\RR_\ell:=\bigcup_{T\in\RR_\ell}T$. Hence,
the estimator $\rho_\ell(\cdot)$ from~\eqref{eq:symm:equivalence} satisfies
\begin{align*}
 (1-q)\,\rho_\ell(F;\RR_\ell)^2
 &= (1-q)\int_{\bigcup\RR_\ell}\widetilde h_\ell|\nabla_\Gamma(F-AU_\ell(F))|^2\,d\Gamma
 \\&
 \le \int_{\Gamma}\widetilde h_\ell|\nabla_\Gamma(F-AU_\ell(F))|^2\,d\Gamma
- \int_{\Gamma}\widetilde h_{\ell+k}|\nabla_\Gamma(F-AU_\ell(F))|^2\,d\Gamma
 \\&
 = \norm{\widetilde h_\ell^{1/2}\nabla_\Gamma(F-AU_\ell(F))}{L^2(\Gamma)}^2
 - \norm{\widetilde h_{\ell+k}^{1/2}\nabla_\Gamma(F-AU_\ell(F))}{L^2(\Gamma)}^2.
\end{align*}
For arbitrary $a,b\ge0$ and $\delta>0$, the Young inequality gives
$(a+b)^2 \le (1+\delta)a^2 + (1+\delta^{-1})b^2$ and hence
$a^2 \ge (1+\delta)^{-1}\big((a+b)^2-(1+\delta^{-1})b^2\big)$. Together with
the triangle inequality, this leads us to
\begin{align*}
 (1-q)\,\rho_\ell(F;\RR_\ell)^2
 \le &\rho_\ell(F)^2 - \frac{1}{1+\delta}\,\rho_{\ell+k}(F)^2
 \\&+ \frac{1+\delta^{-1}}{1+\delta}\,
 \norm{\widetilde h_{\ell+k}^{1/2}\nabla_\Gamma A(U_\ell(F)-U_{\ell+k}(F))}{L^2(\Gamma)}^2.
\end{align*}
Finally, we use an inverse estimate from~\cite[Corollary~3]{afembem} 
\begin{align}\label{eq:novel:invest}
 \norm{h_\ell^{1/2}\nabla_\Gamma AV_\ell}{L^2(\Gamma)}
 \le \Cinv\,\norm{V_\ell}{\H^{-1/2}(\Gamma)}\quad\text{for all }V_\ell\in\PP^0(\TT_\ell).
\end{align}
With this, we derive
\begin{align*}
 (1-q)\,\rho_\ell(F;\RR_\ell)^2
 \le \rho_\ell(F)^2 - \frac{1}{1+\delta}\,\rho_{\ell+k}(F)^2
 + \frac{1+\delta^{-1}}{1+\delta}\,\Cinv\,
 \norm{U_\ell(F)-U_{\ell+k}(F)}{\H^{-1/2}(\Gamma)}^2.
\end{align*}
This proves Assumption~\eqref{ass:contraction} with 
$\c{inv} = \max\{\Cinv,(1-q)^{-1}\}$.

\next 
To see~\eqref{ass:stable}, recall that $A$ is linear and $\mu_\ell(\cdot)$ is always 
efficient~\eqref{eq:twolevel:efficient}. Therefore,~\eqref{ass:stable} follows with 
the abstract arguments of~\eqref{eq:18}.
\end{proof}

\subsection{Faermann's residual error estimator}
\label{section:meshsize}
\noindent
For a given triangulation $\TT_\star$ of $\Gamma$, let 
$\NN_\star$ be the set of nodes of $\TT_\star$. Define the node patch
\begin{align}
 \omega_\star(z) := \bigcup\set{T\in\TT_\star}{z\in T}\subseteq \Gamma,
\end{align}
i.e., the union of all elements which contain $z$. 
The Faermann error estimator was introduced in \cite{faermann2d,faermann3d} for 
$d=2$ resp. $d=3$.
Its local contributions  read
\begin{equation}\label{eq:faermann}
\mu_{\star}(F;T)^{2}:=\sum_{z\in T\cap\NN_{\star}}\semiHh{F-AU_{\star}(F)}{\omega_\star(z)}^{2}\ffor T\in\TT_{\star}.
\end{equation}
Here, $\semiHs{\cdot}{\omega}$, for $0<s<1$, denotes the Sobolev-Slobodeckij
seminorm
\[
\semiHs u{\omega}^{2}=\int_{\omega}\int_{\omega}\frac{|u(x)-u(y)|^{2}}{|x-y|^{d-1+2s}}\, d\Gamma(x)\, d\Gamma(y)\ffor u\in H^{s}(\omega).
\]
So far, the Faermann error estimator is the only a~posteriori BEM error estimator
which is proven to be reliable and efficient~\cite{faermann2d,faermann3d,cf}
\begin{align}\label{eq:faermann0}
 \Ceff^{-1}\,\mu_\star(F) \le \norm{u(F)-U_\star(F)}{\H^{-1/2}(\Gamma)}
 \le \Crel\,\mu_\star(F).
\end{align}
The constants $\Ceff,\Crel>0$ depend only on $\Gamma$ and the shape 
regularity~\eqref{eq:shaperegular} of $\TT_\star$. We note that efficiency of, e.g.,
the weighted-residual error estimator $\eta_\star(\cdot)$ is so far only mathematically
proved for 2D and particular smooth right-hand sides $F$; see~\cite{eps65}.

\begin{theorem}
\label{thm:faermann-convergence}%
Suppose that the Faermann error estimator \eqref{eq:faermann} is used for 
marking~\eqref{eq:doerfler}.
Suppose that the mesh-refinement guarantees uniform $\gamma$-shape
regularity~\eqref{eq:shaperegular} of the\linebreak meshes $\TT_{\ell}$ generated, as well as that
all marked elements $T\in\MM_{\ell}$ are refined into sons $T'\in\TT_{\ell+1}$
with $|T'|\le\kappa\,|T|$ with some uniform constant $0<\kappa<1$.
For all $F\in H^{1/2}(\Gamma)$, Algorithm~\ref{algorithm} then guarantees 
estimator convergence
\begin{align}\label{eq:faermann1}
\mu_{\ell}(F)\to0\quad\text{as }\ell\to\infty
\end{align}
as well as convergence of the discrete solutions
\begin{align}\label{eq:faermann2}
\norm{u(F)-U_{\ell}(F)}{\H^{-1/2}(\Gamma)}\to0\quad\text{as }\ell\to\infty.
\end{align}
\end{theorem}

Note that the convergence~\eqref{eq:faermann2} follows from the estimator 
convergence~\eqref{eq:faermann1} and reliability~\eqref{eq:faermann0}.
Hence, the claim of Theorem~\ref{thm:faermann-convergence} follows from 
Proposition~\ref{prop:new} as soon
as we have verified the abstract assumptions~\eqref{ass:local}--\eqref{ass:stable}.
While the proofs of~\eqref{ass:contraction}--\eqref{ass:stable}
are similar to those of the two-level error estimator from 
Theorem~\ref{thm:symm:twolevel}, the proof of~\eqref{ass:local} is
technically more involved and yields $\mu_\ell(F;\MM_\ell)\lesssim \rho_\ell(F;\RR_\ell)$ with $\RR_\ell$ consisting of all marked elements plus one additional 
layer of elements, i.e.,  
\begin{align}\label{eq:faermann4}
\RR_\ell := \set{T\in\TT_\ell}{\exists T'\in\MM_\ell\quad T\cap T'\neq\emptyset}.
\end{align}

\begin{proof}[Proof of Assumptions \eqref{ass:contraction}--\eqref{ass:stable} for
Theorem \ref{thm:faermann-convergence}.]
In view of~\eqref{eq:faermann4}, we require a modified mesh-width
function $\widetilde{h}_{\ell}:\Gamma\to\R$ which is contractive on each element
$T$ which touches a marked element. For a subset $\EE_{\ell}\subseteq\TT_{\ell}$,
we define the $k$-patch $\omega_\ell^k(\EE_\ell)\subseteq \TT_{\ell}$ inductively by 
\begin{subequations}\label{eq:kpatch}
\begin{align}
\omega_{\ell}^{0}(\EE_{\ell}) & =\EE_{\ell}\quad\text{and}\quad\omega_{\ell}^{k}(\EE_{\ell})=\set{T\in\TT_{\ell}}{\exists T'\in\omega_{\ell}^{k-1}(\EE_{\ell})\quad T\cap T'\neq\emptyset}.
\end{align}
For simplicity, we write 
\begin{align}
 \omega_\ell(\cdot):=\omega_\ell^1(\cdot)
 \quad\text{and}\quad
 \omega_\ell^k(T):=\omega_\ell^k(\{T\})
 \quad\text{for elements $T\in\TT_\ell$}.
\end{align}%
\end{subequations}
Then, there exists $\widetilde{h}_{\ell}:\Gamma\to\R$ which satisfies, for
fixed $k\in\N$ and arbitrary $\ell\in\N$, 
\begin{subequations}\label{eq:mannomann}
\begin{align}
\c{meshsize}^{-1}\diam(T)\le\widetilde{h}_{\ell}|_{T}\le\diam(T) & \quad\text{for all }T\in\TT_{\ell},\label{eq:mesh-width-equivalence}\\
\widetilde{h}_{\ell+1}|_{T}\le\widetilde{h}_{\ell}|_{T} & \quad\text{for all }T\in\TT_{\ell},\\
\widetilde{h}_{\ell+1}|_{T}\le q\,\widetilde{h}_{\ell}|_{T} & \quad\text{for all }T\in\omega_{\ell}^{k}(\TT_{\ell}\backslash\TT_{\ell+1})\label{eq:mesh-width-contractivity}
\end{align}
\end{subequations}
with constants $\c{meshsize}>0$ and $0<q<1$. We note that 
$\TT_{\ell}\backslash\TT_{\ell+1}$
are precisely the refined elements. For bisection-based mesh-refinement in 2D
and 3D, the explicit construction of such a modified
mesh-width function $\widetilde h_\ell$ is given 
in~\cite[Lemma~2]{ffkmp:part1}. In~\cite[Section~8.7]{axioms}, the construction is 
generalized to $\gamma$-shape regular triangulations $\TT_{\ell}$ of 
$n$-dimensional manifolds, $n\ge2$. For $d=2$, i.e.\ $\Gamma$ being a one-dimensional
manifold, the construction is even simpler. 

\next
Overall, we consider an equivalent
weighted-residual error estimator $\rho_\ell(F)$ given by
\begin{align}
 \c{meshsize}^{-1/2}\,\eta_\ell(F;T)
 \le \rho_\ell(F;T) 
 := \norm{\widetilde h_\ell^{1/2}\nabla_\Gamma(F-AU_\ell(F))}{L^2(T)}
 \le \eta_\ell(F;T)
\end{align}
with arbitrary, but fixed $k\geq 1$. 

%
\next To prove \eqref{ass:contraction} with $\RR_{\ell}=\omega_{\ell}(\MM_{\ell})$,
we note that all marked elements are refined, i.e., $\omega_{\ell}^{k}(\MM_{\ell})\subseteq\omega_{\ell}^{k}(\TT_{\ell}\backslash\TT_{\ell+1})$.
Therefore, property \eqref{eq:mesh-width-contractivity} of $\widetilde{h}_{\ell}$ 
ensures $\widetilde{h}_{\ell+1}|_{T}\le q\,\widetilde{h}_{\ell}|_{T}$
for all $T\in\RR_{\ell}$. Arguing as in Theorem~\ref{thm:symm:twolevel}, we
prove \eqref{ass:contraction}.

\next To see~\eqref{ass:stable}, recall that $A$ is linear and $\mu_\ell(\cdot)$ is always 
efficient~\eqref{eq:faermann0}. Therefore,~\eqref{ass:stable} follows with 
the abstract arguments of~\eqref{eq:18}.
\end{proof}

The following proposition provides an estimate for the Slobodeckij seminorm,
needed to establish the local lower bound \eqref{ass:local}. It is related to recent results from~\cite{heuer14}, which studies scalability of different $H^s$-seminorms. Unlike~\cite{heuer14}, we consider node patches 
\begin{align}\label{eq:patch:gme}
 \omega_\star(z) := \bigcup\set{T\in\TT_\star}{z\in T}
\end{align}%
instead of elements.
\def\patch{\omega_\star(z)}

\begin{proposition}
\label{thm:faermann-residual-estimate}Let $\mesh$ be a triangulation
of $\Gamma$, $z\in\nodes$ and $s\in(0,1)$. Then, 
\begin{equation}
\semiHs v{\patch}\le C_{\star}\diam(\patch)^{1-s}\norm{\sgrad{\Gamma}v}{L^{2}(\patch)}\ffor v\in H^{1}(\patch).\label{eq:slobodeckij-estimate}
\end{equation}
The constant $C_{\star}>0$ depends only on $\Gamma$ and the $\gamma$-shape
regularity of $\mesh$.
\end{proposition}

We postpone the proof of Proposition~\ref{thm:faermann-residual-estimate} and show how it implies \eqref{ass:local} for all $F\in H^1(\Gamma)$.

\begin{proof}[Proof of Assumption \eqref{ass:local} for Theorem \ref{thm:faermann-convergence}.]
Let $T\in\TT_{\ell}$. Summing \eqref{eq:slobodeckij-estimate} over
$z\in\NN_{\ell}\cap T$, we get
\[
\sum_{z\in\NN_{\ell}\cap T}\semiHs{v}{\omega_\ell(z)}^{2}\lesssim\sum_{z\in\NN_{\ell}\cap T}\diam(\omega_\ell(z))^{2(1-s)}\norm{\sgrad{\Gamma}v}{L^{2}(\omega_\ell(z))}^{2}.
\]
For $s=1/2$, $v=F-AU_{\ell}(F)\in H^1(\Gamma)$, and $\omega:=\bigcup\omega_\ell(\MM_\ell)=\bigcup\RR_\ell$ (see~\eqref{eq:kpatch} for the definition of the patch), this shows
\begin{eqnarray*}
\mu_{\ell}(F;\MM_{\ell})^{2} & = & \sum_{T\in\MM_{\ell}}\sum_{z\in\NN_{\ell}\cap T}
|F-AU_{\ell}(F)|_{H^{1/2}(\omega_\ell(z))}^{2}\\
 & \lesssim & \sum_{T\in\MM_{\ell}}\sum_{z\in\NN_{\ell}\cap T}\diam(T)\norm{\sgrad{\Gamma}(F-AU_{\ell}(F))}{L^{2}(\omega_\ell(z))}^{2}\\
 & \simeq & \norm{\widetilde{h}_{\ell}^{1/2}\sgrad{\Gamma}(F-AU_{\ell}(F))}{L^{2}(\omega)}^{2}=\rho_{\ell}(F;\RR_{\ell})^{2}.
\end{eqnarray*}
This concludes the proof.
\end{proof}

To establish Proposition~\ref{thm:faermann-residual-estimate}, we
need two additional lemmas. The first enables us to use a ``generalized'' scaling
argument which allows for bi-Lipschitz deformations of the reference domain. A mapping $\kappa:O\to\R^d$ with $O\subset \R^k$ open and $1\leq k\leq d$ is called bi-Lipschitz if it satisfies for some constants $L_1,L_2>0$
\begin{equation}\label{eq:bilipschitz}
L_1|x-y|\le|\k(x)-\k(y)|\le L_2|x-y|\ffor x,y\in O.
\end{equation}
This allows to formulate the following lemma.

\begin{lemma}[Generalized scaling property of Sobolev seminorms]
\label{lem:scaling} Let $\k:O\to\R^{d}$ be bi-Lipschitz~\eqref{eq:bilipschitz}. Then,
it holds
\begin{subequations}
\begin{equation}
C^{-1}L_1^{k/2-s}\semiHs{v\circ\kappa}O\le\semiHs{v}{\kappa(O)}\le CL_2^{k/2-s}\semiHs{v\circ\kappa}O\label{eq:scaling}
\end{equation}
for all $v\in H^{s}(\kappa(O))$ and $0<s\le1$. The constant $C>0$ satisfies
\begin{align}\label{eq:scaling2}
C\le (L_2/L_1)^{(d+2)/2}.
\end{align}
\end{subequations}
\end{lemma}

\begin{proof}
First, we consider the case $0<s<1$. According to Rademacher's theorem~\cite[Section~3.1]{evans92}, Lipschitz continuous functions are differentiable almost everywhere. An immediate consequence of~\eqref{eq:bilipschitz}
thus is
\begin{equation}
L_{1}|v|\le|D\k(x)v|\le L_{2}|v|\ffor v\in\R^{k}\text{ and a.e. }x\in O.\label{eq:bilipschitz2}
\end{equation}
Denote the Jacobian determinant by $J\kappa:=\sqrt{\det(D\kappa (D\kappa)^T)}$.
Interpreting~\eqref{eq:bilipschitz2} as an estimate for the eigenvalues of 
$D\kappa (D\kappa)^T$, one obtains
\begin{equation}
L_{1}^{k}\le J\k\le L_{2}^{k}\quad\text{a.e. in }O\label{eq:bilipschitz3}.
\end{equation}%
The estimates~\eqref{eq:bilipschitz3} and~\eqref{eq:bilipschitz} show
\begin{eqnarray*}
\semiHs{v}{\k(O)}^{2} & = & \int_{O}\int_{O}\frac{|v\circ\kappa(x)-v\circ\kappa(y)|^{2}}{|\k(x)-\k(y)|^{k+2s}}J\k(x)J\k(y)\, dx\, dy\\
 & \le & L_{2}^{2k}\int_{O}\int_{O}\frac{|v\circ\kappa(x)-v\circ\kappa(y)|^{2}}{|\k(x)-\k(y)|^{k+2s}}\, dx\, dy\\
 & \le & L_{1}^{-(k+2s)}L_{2}^{2k}\int_{O}\int_{O}\frac{|v\circ\kappa(x)-v\circ\kappa(y)|^{2}}{|x-y|^{k+2s}}\, dx\, dy
 = L_1^{-(k+2s)}L_2^{2k} \semiHs{v\circ\kappa}{O}^2.
\end{eqnarray*}
This proves $\semiHs{v}{\k(O)} \leq (L_2/L_1)^{k/2+s}L_2^{k/2-s}\semiHs{v\circ\kappa}{O}$. With $(L_2/L_1)\geq 1$ and $k/2+s\le(d+2)/2$, we obtain the upper estimate of~\eqref{eq:scaling}.
The lower estimate
follows analogously.

The case $s=1$ follows from the chain rule and~\eqref{eq:bilipschitz2}--\eqref{eq:bilipschitz3}, where $\Gamma:=\kappa(O)$ is the induced surface: The pointwise estimate 
\[
L_{2}^{-2}|\nabla (v\circ\kappa)|^{2}\le|(\nabla_\Gamma v)\circ\k|^{2}\le L_{1}^{-2}|\nabla (v\circ\kappa)|^{2}
\]
and integration over $O$ shows
\[
L_2^{-2}L_{1}^{k}\semiH{v\circ\kappa}O^{2}\le\semiH{v}{\k(O)}^{2}\le L_1^{-2}L_{2}^{k}\semiH{v\circ\kappa}O^{2}.
\]
This concludes the proof for $s=1$.
\end{proof}

\begin{figure}[t]
\includegraphics[width=.23\textwidth]{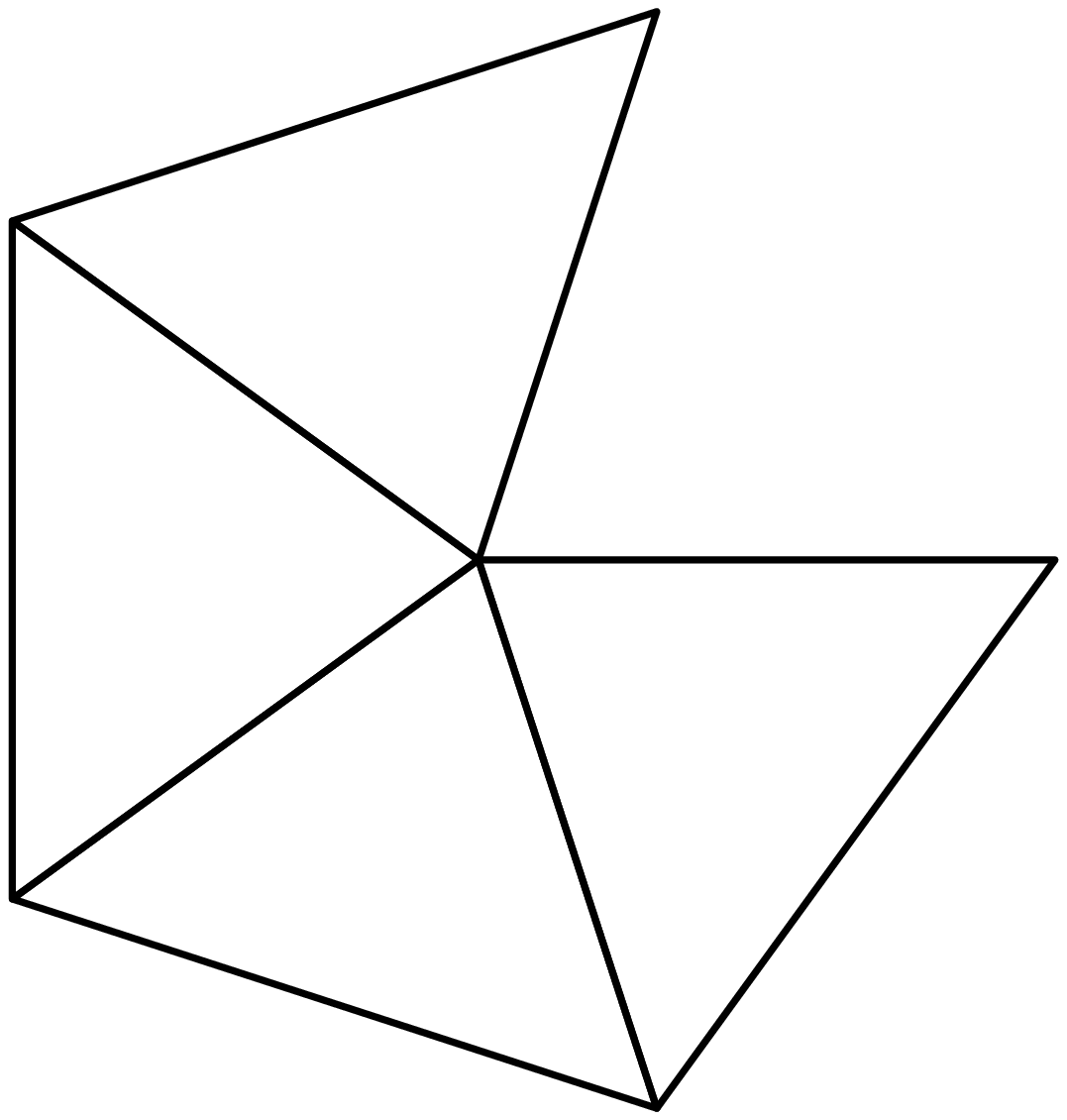}
\hfill
\includegraphics[width=.23\textwidth]{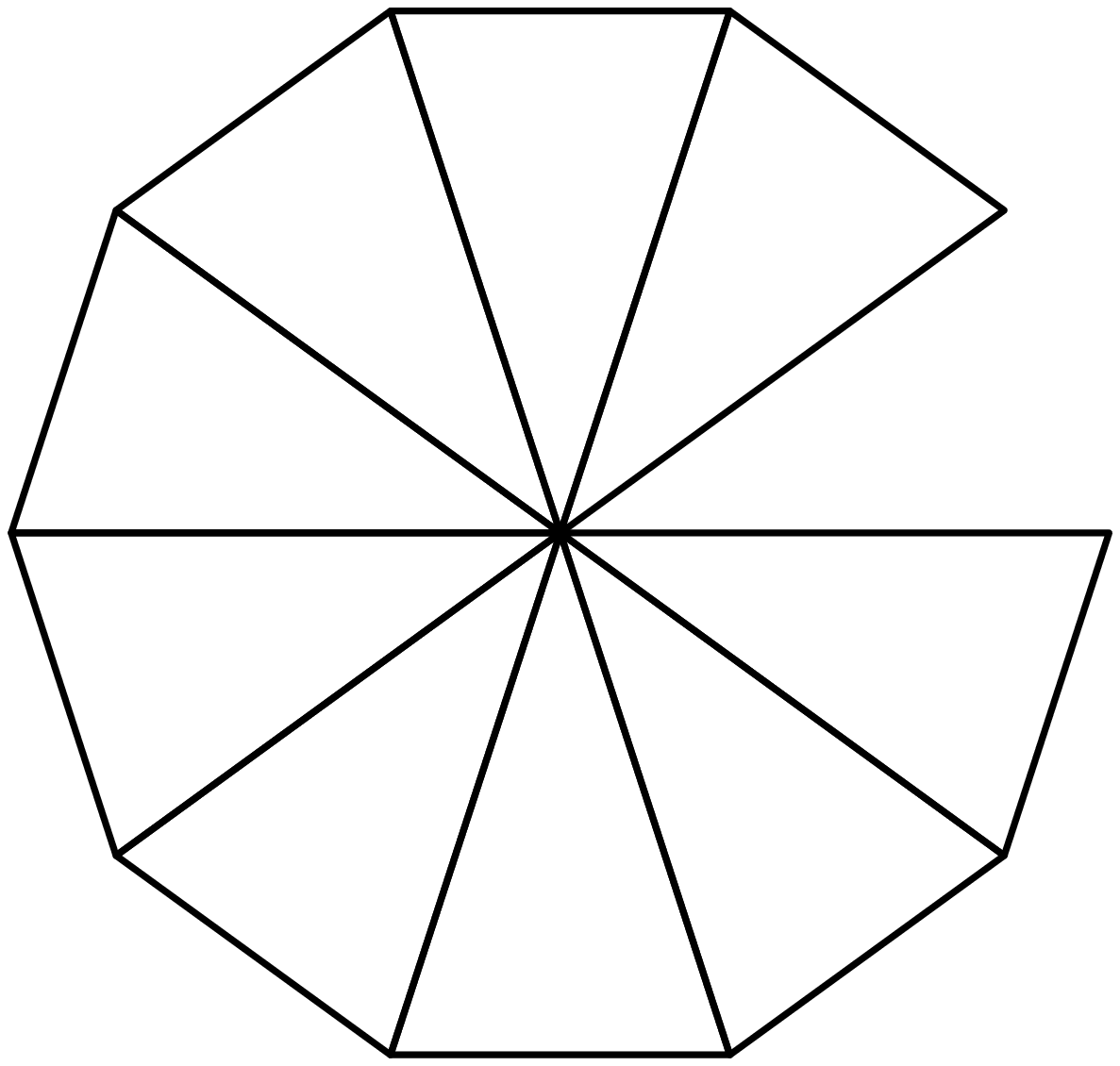}
\hfill
\includegraphics[width=.23\textwidth]{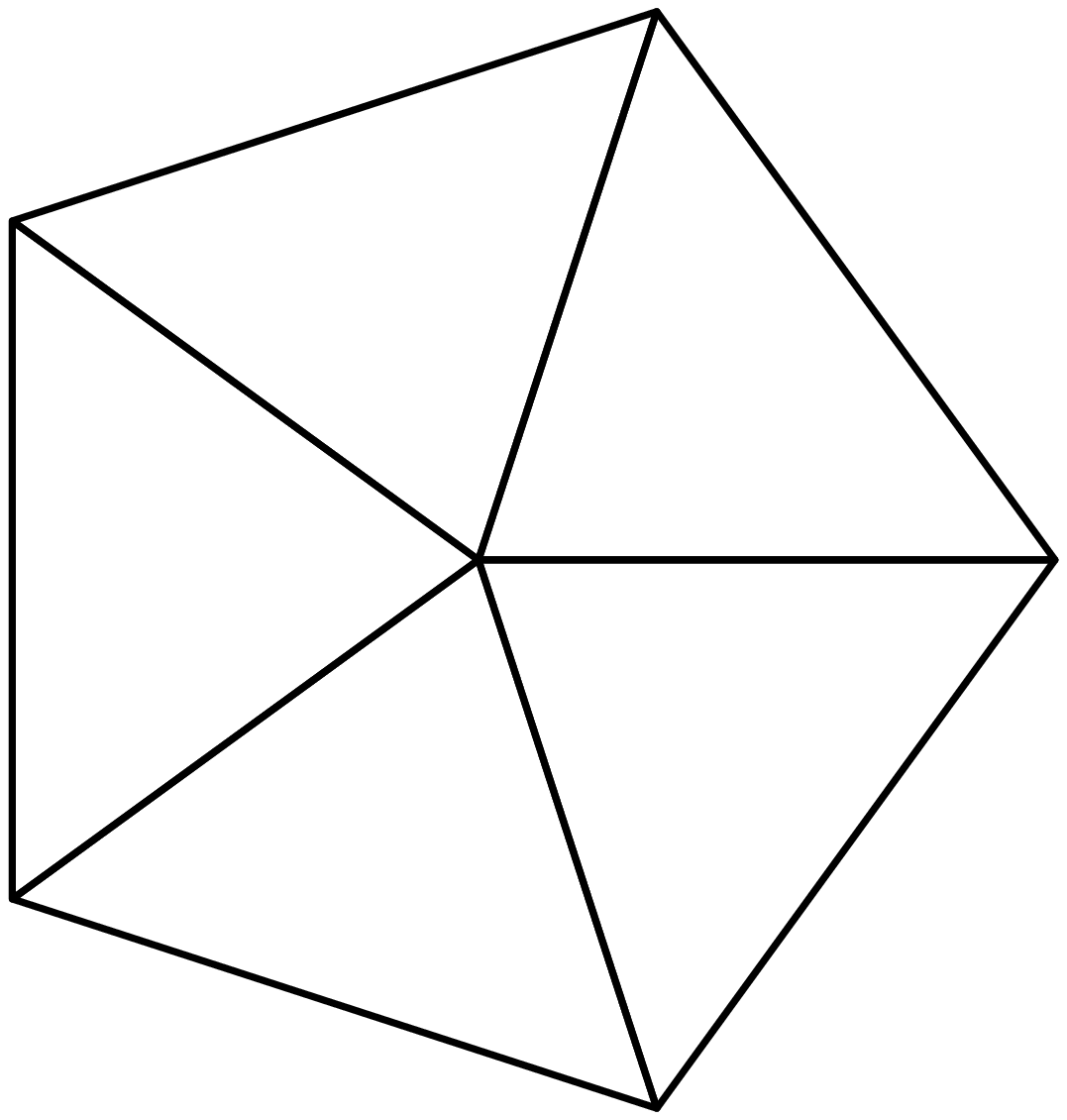}
\hfill
\includegraphics[width=.23\textwidth]{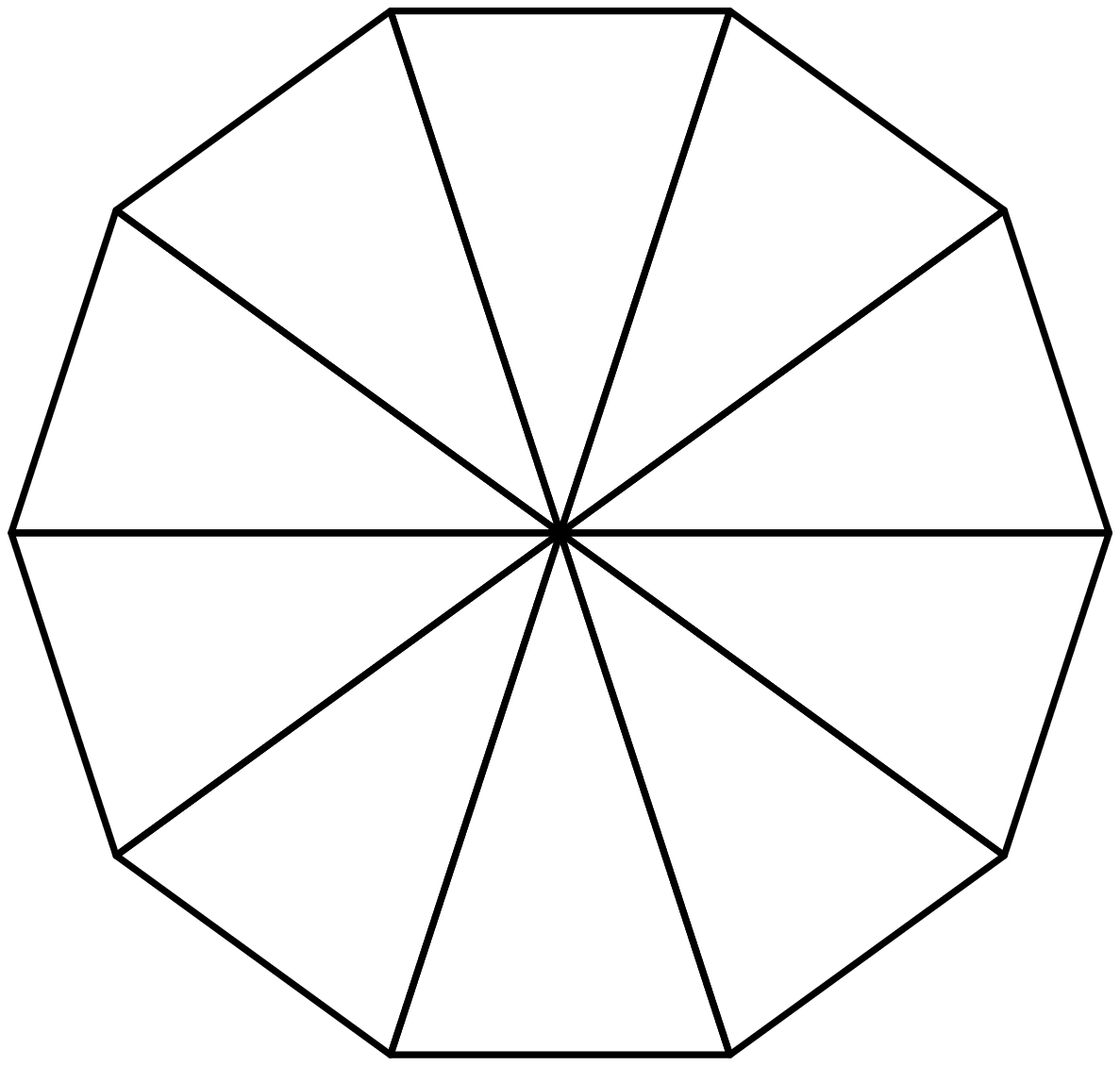}
\caption{Reference patches $\widehat\omega_N$ and $\widehat\omega'_N$ in 
Lemma~\ref{lem:mapping} for $z\in\partial\Gamma$ and $N=4$ as well as $N=9$ 
resp.\ for $z\not\in\partial\Gamma$ and $N=5$ as well as $N=10$ (from left to right).}
\label{fig:patches}
\end{figure}

It remains to bound the Lipschitz constants $(L_2/L_1)^{(d+2)/2}$ in~\eqref{eq:scaling2} for our particular case of $\kappa(O)$ being a node-patch on a polyhedral surface.
To that end, define for any $N\ge3$ the \emph{reference patch }$\patchref N\subset\R^{2}=\mathbb C$
to be the compact regular polygon with corners $e^{\frac{2\pi ik}{N}}$,
for $k=0,\ldots,N-1$ (where $0$ is an interior point).
Moreover, let ${\rm conv}\{\cdot\}$ denote the closed convex hull.
Define $\widehat\omega_1^\prime:={\rm conv}\{0,1,i\}$ and,
for $N\geq 2$, $\widehat\omega_N^\prime:=\widehat\omega_{N+1} \setminus{\rm interior}({\rm conv}\{0,1,e^{\frac{2\pi i}{N+1}}\})$ (where $0$ is a boundary vertex);
see Figure~\ref{fig:patches}.
The next lemma constructs appropriate uniformly bi-Lipschitz pullbacks to the reference patches.
Since the proof is elementary but lengthy, we only sketch it and refer 
to~\cite{mitschaeibl} for the details.

\begin{lemma}
\label{lem:mapping}Let $z\in\nodes$ be some node of a triangulation $\mesh$ of $\Gamma\subset\R^d$, and let $d=2,3$. Let $N:=\#\{T\in\mesh\colon z\in T\}$
be the number of elements in the node patch $\patch$ from~\eqref{eq:patch:gme} and define
\begin{align*}
\widehat\omega := (-1,1)
\quad\text{for $d=2$}\quad
\text{resp.}\quad
\widehat\omega := \begin{cases}
\widehat\omega_N^\prime&\text{for }z\in\partial\Gamma,\\
\widehat\omega_N&\text{for }z\not\in\partial\Gamma,
\end{cases}
\quad\text{for $d=3$.}
\end{align*}%
Then, there exists 
\[
\k_{z}:\widehat\omega\to\patch
\]
bi-Lipschitz with
\begin{eqnarray}\label{eq:bilip:kappa}
C^{-1}\, \diam(\patch)\le L_1\quad
\text{and}
\quad
L_2  \le  C\,\diam(\patch).
\end{eqnarray}
The constant $C>0$ depends only on $\Gamma$ and the $\gamma$-shape
regularity of $\mesh$.
\end{lemma}
%
\begin{proof}[Sketch of proof]
We only sketch the case $d=3$, whereas the simpler case $d=2$ is left to the reader.
Let $\widehat T_1,\ldots,\widehat T_N$ denote the elements in $\widehat\omega=\bigcup_{j=1}^N\widehat T_j$ and let $T_1,\ldots,T_N$ denote the elements of $\omega_\star(z)=\bigcup_{j=1}^N T_j$.
Without loss of generality, we assume that the numbering of the elements is such that $\#(\widehat T_i\cap\widehat T_j) =\#(T_i\cap T_j)\in\{1,\infty\}$ for all $1\leq i,j\leq N$.
This allows to find a unique affine mapping $\kappa_j:\,\widehat T_j\to T_j$ which satisfies
\begin{align*}
 \kappa_j(0)=z\quad\text{and}\quad\kappa_j(\widehat T_j\cap \widehat T_i)= T_j\cap T_i\quad \text{for all }i=1,\ldots,N.
\end{align*}
Define $\kappa:\,\widehat\omega\to\omega_\star(z)$ as
\begin{align*}
 \kappa|_{\widehat T_j} = \kappa_j\quad\text{for all }j=1,\ldots,N.
\end{align*}
If $z^\prime\in (\widehat T_j\cap \widehat T_i)\setminus\{0\}$, there holds $\kappa_j(z^\prime)\in T_i\cap T_j$ and $\kappa_i(z^\prime)\in T_i\cap T_j$ by definition.
Since the $\kappa_j$ are affine, there holds 
$\kappa_i|_{\widehat E} = \kappa_j|_{\widehat E}$ on $\widehat E=\widehat T_i\cap\widehat T_j$.
This shows that $\kappa$ is well-defined and continuous.
Straightforward arguments show that $N$ and the Lipschitz continuity of the $\kappa_j$ depend only on the $\gamma$-shape regularity of $\TT_\star$. 
The Lipschitz continuity~\eqref{eq:bilip:kappa} of $\kappa$ depends additionally 
on $\Gamma$.
\end{proof}

With this at hand, the proof of Proposition~\ref{thm:faermann-residual-estimate} follows.

\begin{proof}[Proof of Proposition \ref{thm:faermann-residual-estimate}]
Using the mapping $\k=\k_z$ from Lemma \ref{lem:mapping}, we can apply
Lem\-ma~\ref{lem:scaling} with $O=\widehat\omega$ and $\k(O)=\patch$.
This immediately gives 
\[
\semiHs{v}{\patch}
\simeq\diam(\patch)^{(d-1)/2-s}\semiHs{v\circ\kappa}{\widehat\omega}
\]
for all $v\in H^{1}(\patch)$, $s\in(0,1]$, with constants depending only
on the $\gamma$-shape regularity of $\TT_\ell$. On the reference patch, we can
use the continuous embedding $H^{1}(\widehat\omega)\subset H^{s}(\widehat\omega)$
and Poincar\'e's inequality to obtain
\[
\semiHs{v\circ\kappa}{\widehat\omega}
= \min_{c\in\R}|v\circ\kappa-c|_{H^s(\widehat\omega)}
\le \min_{c\in\R}\|v\circ\kappa-c\|_{H^s(\widehat\omega)}
\lesssim \min_{c\in\R}\|v\circ\kappa-c\|_{H^1(\widehat\omega)}
\lesssim\semiH{v\circ\kappa}{\widehat\omega}.
\]
The hidden constant depends only on $\widehat\omega$ and is hence controlled by the $\gamma$-shape regularity of $\TT_\ell$. Combining the last
two estimates, we get
\begin{eqnarray*}
\semiHs{v}{\patch} & \simeq & \diam(\patch)^{(d-1)/2-s}\semiHs{v\circ\kappa}{\widehat\omega}\\
 & \lesssim & \diam(\patch)^{(d-1)/2-s}\semiH{v\circ\kappa}{\widehat\omega}\\
 & \simeq & \diam(\patch)^{1-s}\semiH{v}{\patch}.
\end{eqnarray*}
This concludes the proof.
\end{proof}

\subsection{Remarks and Extensions}
%
The inverse estimates of~\cite[Theorem~3.6]{ghs} and~\cite[Corollary~3]{afembem}
also apply to higher-order discretizations $\PP^p(\TT_\star)$ with piecewise 
polynomials of degree $p\ge0$ and curved surface triangles (where $\Gamma$ is assumed to be piecewise smooth). Also 
Proposition~\ref{thm:faermann-residual-estimate} can be proved for non-polygonal
boundaries. Consequently, the convergence results of Theorem~\ref{thm:symm:twolevel}
and Theorem~\ref{thm:faermann-convergence} also transfer to these settings. 
Moreover, rectangular elements can be covered.

In~\cite{faermann2d} the spaces $\PP^p(\TT_\star)$ are defined by local pullback
with the arc-length para\-metri\-za\-tion. While this is immaterial for piecewise affine
boundaries, $\PP^p(\TT_\star)$ depends on the chosen parametrization for
non-affine boundaries. For 2D BEM, this restriction is removed in the recent work~\cite{igabem}.

\def\normal{\boldsymbol{n}}

\section{Hyper-singular integral equation}
\label{section:hypsing}

\subsection{Model problem}
We consider the hyper-singular integral equation
\begin{align}\label{eq:hypsing}
 Au(x) = -\partial_{\normal(x)}\int_\Gamma \partial_{\normal(y)}G(x-y)\,u(y)\,d\Gamma(y) = F(x)
 \quad\text{for all }x\in\Gamma
\end{align}
on a relatively open, connected, and polygonal part $\Gamma\subsetneqq\partial\Omega$ of the boundary of a 
bounded, polyhedral Lipschitz domain $\Omega\subset\R^d$, $d=2,3$. 
(The case $\Gamma = \partial\Omega$ is sketched in Section~\ref{section:hypsing:remarks} below.) 
For $d=3$, we assume that the boundary of $\Gamma$ (a polygonal curve) is Lipschitz itself. In~\eqref{eq:hypsing}, $G$ denotes the 
fundamental solution of the Laplacian; see~\eqref{eq:G}. Moreover, 
$\partial_{\normal(x)}$ denotes the normal derivative at $x\in\Gamma$ with $\normal(x)$ the outer unit normal vector of $\Omega$.
The reader is referred to, e.g., the monographs~\cite{hw,mclean,ss,steinbach} for
proofs of and details on the following facts:
The hyper-singular integral operator
$A:\HH\to\HH^*$ is a continuous linear operator between the fractional-order 
Sobolev space $\HH=\H^{1/2}(\Gamma)$ and its dual $\HH^*=H^{-1/2}(\Gamma)$. Duality is 
understood with respect to the extended $L^2(\Gamma)$-scalar product $\dual\cdot\cdot$. 
Then, the hyper-singular integral operator is also elliptic
\begin{align}
 \dual{Av}{v} \ge \c{elliptic}\,\norm{v}{\H^{1/2}(\Gamma)}^2
 \quad\text{for all }v\in\HH = \H^{1/2}(\Gamma) 
\end{align}
with some constant $\c{elliptic}>0$ which depends only on $\Gamma$.
Thus, $A$ meets all assumptions of Section~\ref{section:abstract}, 
and $\norm{v}A^2:=\dual{Av}{v}$ even defines an equivalent Hilbert norm
on $\HH$.

\subsection{Discretization}
Let $\TT_\star$ be a $\gamma$-shape regular triangulation of $\Gamma$ as defined in Section~\ref{section:discretization}.
With $\XX_\star=\widetilde\SS^1(\TT_\star):=\PP^1(\TT_\star)\cap \H^{1/2}(\Gamma)$ being the space of $\TT_\star$-piecewise
affine, globally continuous functions which vanish at the boundary of $\Gamma$, we now consider the Galerkin formulation~\eqref{eq:galerkin}.

\subsection{Weighted-residual error estimator}
For given right-hand side $F\in L^2(\Gamma)$, the
residual $F-AU_\star(F)\in H^{-1/2}(\Gamma)$ has additional regularity
$F-AU_\star\in L^2(\Gamma)$,
since $A:\H^{1/2+s}(\Gamma)\to H^{-1/2+s}(\Gamma)$ is stable for 
$-1/2 \le s \le 1/2$ (but not isomorphic for $s=\pm1/2$).
It is proved in~\cite{cmps} that 
\begin{align}\label{eq:hypsing:eta}
 \norm{u(F)-U_\star(F)}{\H^{1/2}(\Gamma)}
 \simeq \norm{F-AU_\star(F)}{H^{-1/2}(\Gamma)}
 \lesssim \norm{h_\star^{1/2}(F-AU_\star(F))}{L^2(\Gamma)}
 =:\eta_\star(F).
\end{align}
Overall, this proves the reliability estimate
\begin{align}
 \norm{u(F)-U_\star(F)}{\H^{1/2}(\Gamma)}
 \le \CCrel\,\eta_\star(F),
\end{align}
and the constant $\CCrel>0$ depends only on $\Gamma$
and the $\gamma$-shape regularity~\eqref{eq:shaperegular} of $\TT_\star$.
In particular, the weighted-residual error estimator can be localized via
\begin{align}
 \eta_\star(F) = \Big(\sum_{T\in\TT_\star}\eta_\star(F;T)^2\Big)^{1/2}
 \text{ with }
 \eta_\star(F;T) = \diam(T)^{1/2}\norm{F-AU_\star(F)}{L^2(T)}.
\end{align} 
Recently, convergence of Algorithm~\ref{algorithm} has been shown 
even with quasi-optimal rates, if $\eta_\ell(F)=\mu_\ell(F)$ is used for 
marking~\eqref{eq:doerfler}, see~\cite{gantumur,ffkmp:part2}. We stress that our approach
with $\eta_\ell(F) = \rho_\ell(F) = \mu_\ell(F)$ would also give convergence 
$\eta_\ell(F) \to 0$ as $\ell\to\infty$. Since this is, however, a much weaker
result than that of~\cite{gantumur}, we omit the details.

\subsection{Two-level error estimator}
\label{section:hypsing:twolevel}
Let $\widehat\TT_\star$ denote the uniform refinement of $\TT_\star$.
Let $\widehat\NN_\star$ be the corresponding set of nodes.
Let $z_{T,j}\in T\cap\widehat\NN_\star$, $j=1,\ldots,D$ denote the new nodes of
the uniform refinement $\widehat \TT_\star$ within $T$.
 Let $\{v_{T,1},\dots,v_{T,D}\}\subset \SS^1(\widehat\TT_\star)$ denote the fine-mesh hat functions
which satisfy $v_{T,j}(z_{T,j})=1$ and $v_{T,j}(z)=0$ for all 
$z\in\widehat\NN_\star\backslash\{z_{T,j}\}$.
 We note 
that (in dependence of the chosen mesh-refinement) usually $D=1$ for $d=2$ and $D=3$ for $d=3$. 
In this setting, the two-level error estimator has first been proposed
by~\cite{ms}. 
Its local contributions read
\begin{align}\label{eq:hypsing:twolevel}
\mu_\star(F;T)^2 = \sum_{j=1}^D\mu_{\star,j}(F;T)^2
\quad\text{with}\quad
\mu_{\star,j}(F;T) = \begin{cases}
\frac{\dual{F-AU_\star(F)}{v_{T,j}}}{\dual{Av_{T,j}}{v_{T,j}}^{1/2}}
&\text{for }z_{T,j}\notin\partial\Gamma,\\
0&\text{otherwise}.
\end{cases}
\end{align}
Put differently, we test the residual $F-AU_\star(F)\in H^{-1/2}(\Gamma)$ 
with the additional basis functions from 
$\widetilde \SS^1(\widehat\TT_\star)\backslash\widetilde \SS^1(\TT_\star)$.
This quantity is appropriately scaled by the corresponding energy norm 
$\norm{v_{T,j}}{H^{-1/2}(\Gamma)}\simeq\dual{Av_{T,j}}{v_{T,j}}^{1/2}=\norm{v_{T,j}}A$.
Note that unlike the weighted-residual error estimator $\eta_\star(\cdot)$,  the two-level
error estimator $\mu_\star(F)$ is well-defined under minimal regularity 
$F\in H^{-1/2}(\Gamma)$ of the given right-hand side.

The two-level estimator $\mu_{\star}(\cdot)$ is known to be efficient~\cite{ms,mms,hms,eh,efgp,hypsing3d}
\begin{align}\label{eq:twolevel:efficient:hyp}
 \mu_\star(F) \le \Ceff\,\norm{u(F)-U_\star(F)}{\H^{1/2}(\Gamma)},
\end{align}
while reliability
\begin{align}\label{eq:twolevel:reliable:hyp}
 \norm{u(F)-U_\star(F)}{\H^{1/2}(\Gamma)} \le \Crel\,\mu_\star(F) 
\end{align}
holds under~\cite{ms,mms,hms,eh} and is even equivalent to~\cite{efgp,hypsing3d} the saturation assumption
\begin{align}\label{eq:hypsing:saturation}
\norm{u(F)-\widehat U_\star(F)}A \le q_{\rm sat}\,\norm{u(F)-U_\star(F)}A
\end{align}
in the energy norm $\norm\cdot{A}\simeq\norm\cdot{\H^{1/2}(\Gamma)}$.
Here, $0<q_{\rm sat}<1$ is a uniform constant, and $\widehat U_\star(F)$ is the Galerkin
solution with respect to the uniform refinement $\widehat\TT_\star$ of $\TT_\star$.
The constant $\Ceff>0$ depends only on $\Gamma$ and $\gamma$-shape regularity of
$\TT_\star$, while $\Crel>0$ additionally depends on the saturation constant 
$q_{\rm sat}$.
(The saturation assumption~\eqref{eq:hypsing:saturation} for the $\H^{1/2}$-norm
$\norm\cdot{A} = \norm\cdot{\H^{1/2}/(\Gamma)}$ implies 
reliability~\eqref{eq:twolevel:reliable:hyp}, but is \emph{not} necessary though.)


\begin{theorem}\label{thm:hypsing:twolevel}
Suppose that the two-level error estimator~\eqref{eq:hypsing:twolevel} is used
for marking~\eqref{eq:doerfler}. Suppose that the mesh-refinement
guarantees uniform $\gamma$-shape regularity of the meshes $\TT_\ell$ generated,
as well as that all marked elements $T\in\MM_\ell$ are refined into sons 
$T'\in\TT_{\ell+1}$ with $|T'|\le \kappa\,|T|$ with some uniform constant $0<\kappa<1$.
Then, Algorithm~\ref{algorithm} guarantees 
\begin{align}\label{hypsing:convergence}
 \mu_\ell(F)\to0 
 \quad\text{as }\ell\to\infty.
\end{align}
for all $F\in H^{1/2}(\Gamma)$.
\end{theorem}

\begin{proof}
With 
Proposition~\ref{prop:new}, it remains to verify the abstract assumptions~\eqref{ass:local}--\eqref{ass:stable}.

\next 
We use the modified mesh-width function $\widetilde h_\ell$ from the proof of 
Theorem~\ref{thm:faermann-convergence} and define the modified weighted-residual
error estimator
\begin{align}\label{eq:hypsing:wuerg}
 \c{meshsize}^{-1/2}\,\eta_\ell(F;T)
 \le \rho_\ell(F;T) 
 := \norm{\widetilde h_\ell^{1/2}(F-AU_\ell(F))}{L^2(T)}
 \le \eta_\ell(F;T).
\end{align}
Arguing analogously to the proof of 
Theorem~\ref{thm:symm:twolevel}, we verify contraction~\eqref{ass:contraction}.
The only difference is that instead of~\eqref{eq:novel:invest}, we  use the
inverse-type
estimate
\begin{align}\label{eq:novel:invest:2}
 \norm{h_\ell^{1/2}AV_\ell}{L^2(\Gamma)}
 \le \Cinv\, \norm{V_\ell}{\H^{1/2}(\Gamma)}
 \text{for all }V_\ell\in\widetilde\SS^1(\TT_\ell),
\end{align}
where the constant $\Cinv>0$ depends only on $\Gamma$ and $\gamma$-shape regularity
of $\TT_\ell$; see~\cite[Corollary~3]{afembem}.

\next 
It is proved in~\cite[Theorem~5.4]{cmps} that
\begin{align*}
 \mu_{\ell,j}(F;T) \lesssim \norm{h_\ell^{1/2}(F-AU_\ell)}{L^2(\supp(v_{T,j}))},
\end{align*}
where the hidden constant depends only on $\Gamma$ and $\gamma$-shape regularity of $\TT_\ell$. By definition~\eqref{eq:hypsing:eta} 
of the weighted-residual error estimator and~\eqref{eq:hypsing:wuerg}, this implies 
\begin{align*}
 \mu_\ell(F;T)^2 \lesssim \sum_{{T'\in\TT_\ell}\atop{T'\cap T\neq\emptyset}}\eta_\ell(F;T)^2
 \simeq\sum_{{T'\in\TT_\ell}\atop{T'\cap T\neq\emptyset}}\rho_\ell(F;T)^2.
\end{align*}
Using the notation from the proof of Theorem~\ref{thm:faermann-convergence}, this 
yields~\eqref{ass:local} with $\RR_\ell := \omega_\ell(\MM_\ell)$ being the
marked elements plus one additional layer of elements;
see~\eqref{eq:kpatch} for the definition of $\omega_\ell(\cdot) = \omega_\ell^1(\cdot)$.

\next 
Finally, stability~\eqref{ass:stable} follows from efficiency~\eqref{eq:twolevel:efficient:hyp}; see~\eqref{eq:18}.
\end{proof}

\subsection{Remarks and Extensions}
\label{section:hypsing:remarks}
%
The inverse estimate~\eqref{eq:novel:invest:2} of~\cite[Corollary~3]{afembem} also applies to higher-order 
discretizations $\widetilde\SS^p(\TT_\star):=\PP^p(\TT_\star)\cap\H^{1/2}(\Gamma)$ 
with piecewise polynomials of degree $p\ge1$ and curved surface triangles. 
Consequently, the convergence results of Theorem~\ref{thm:symm:twolevel}
and Theorem~\ref{thm:faermann-convergence} also transfer to these settings. 
Moreover, also rectangular elements can be covered.

If the boundary $\Gamma$ is closed, i.e.\ $\Gamma=\partial\Omega$, the hypersingular
operator $W:H^{1/2}_0(\Gamma)\to H^{-1/2}_0(\Gamma)$ is well-defined and elliptic,
where $H^{\pm1/2}_0(\Gamma)=\set{v\in H^{\pm1/2}(\Gamma)}{\dual{v}{1} = 0}$.
Therefore, well-posedness of~\eqref{eq:hypsing} requires the compatibility condition
$F\in H^{-1/2}_0(\Gamma)$. On the one hand, one may formulate the weak formulation 
of~\eqref{eq:hypsing} as well as its Galerkin discretization with respect to the 
subspaces $\HH=H^{1/2}_0(\Gamma)$ and $\XX_\star = \PP^p(\TT_\star)\cap H^{1/2}_0(\Gamma)$.
On the other hand, one can choose the full space $\HH = H^{1/2}(\Gamma)$ 
and $\XX_\star = \PP^p(\TT_\star)\cap H^{1/2}(\Gamma)$
and consider
the naturally stabilized formulation
\begin{align}\label{eq:hypsing:stabilized}
 a(u,v) := \dual{Au}{v} + \dual{u}{1}\dual{v}{1} = \dual{F}{v}
 \quad\text{for all }v\in \HH = H^{1/2}(\Gamma).
\end{align}
The compatibility condition on $F$ and 
$1\in\SS^1(\TT_\star) = \PP^1(\TT_\star)\cap H^{1/2}(\Gamma)$ ensure that
both, the exact solution $u=u(F)\in H^{1/2}(\Gamma)$ of~\eqref{eq:hypsing:stabilized}
as well as the Galerkin approximation $U_\star=U_\star(F)\in\SS^1(\TT_\star)$, satisfy
$\dual{u(F)}{1}=0=\dual{U_\star(F)}{1}$, i.e., $u(F)\in H^{1/2}_0(\Gamma)$ as well as
$U_\star(F) \in \PP^1(\TT_\star)\cap H^{1/2}_0(\Gamma)$. In either case, the 
weighted-residual error estimator coincides with~\eqref{eq:hypsing:eta} and the
two-level error estimator is obtained analogously to Section~\ref{section:hypsing:twolevel}. 
For the two-level error estimator, we refer, e.g., to~\cite{efgp} for the $H^{1/2}_0(\Gamma)$-based discretization and to~\cite{hypsing3d} for the stabilized approach.
In any case, Theorem~\ref{thm:hypsing:twolevel} holds 
accordingly.

%
\def\UU{\boldsymbol{U}}
\def\uu{\boldsymbol{u}}
\def\vv{\boldsymbol{v}}
\def\gg{\boldsymbol{g}}
\def\ff{\boldsymbol{f}}
\def\ww{\boldsymbol{w}}
\def\VV{\boldsymbol{V}}
\def\WW{\boldsymbol{W}}

\def\proj{P} 
\def\linhull{\mathrm{span}}

\def\revtfblue#1{{\color{blue}{#1}}}

\section{FEM-BEM Coupling}
\label{section:fembem}

\subsection{Model problem}
Let $\Omega \subset \R^d$ be a Lipschitz domain with polygonal boundary $\Gamma := \partial
\Omega$, $d=2,3$.
Let $B:\R^d\to\R^d$ be Lipschitz continuous
\begin{align}\label{eq:fembem:lip}
|Bx - By| \le \c{lipB} |x-y|
\quad\text{for all }x,y\in\R^d
\end{align}
for some $\setc{lipB}>0$.
In addition, we assume that the induced operator
$B : L^2(\Omega)^d \to L^2(\Omega)^d$, $(B\ff)(x) := B(\ff(x))$ is strongly monotone
\begin{align}\label{eq:fembem:mon}
  \int_\Omega (B\ff -B\gg)\cdot(\ff-\gg) \,d\Omega \geq \c{monB} \norm{\ff-\gg}{L^2(\Omega)}^2 \quad\text{for all }
  \ff,\gg\in L^2(\Omega)^d
\end{align}
with monotonicity constant $\setc{monB}>1/4$. 
(Arguing as in~\cite{os}, this assumption can be sharpened to 
$\c{monB}>q_{\dlp}/4$, where $1/2\le q_{\dlp}<1$ is the contraction constant of
the double-layer integral operator.)
We consider a possibly nonlinear 
Laplace transmission 
problem which is reformulated in terms of the Johnson-N\'ed\'elec FEM-BEM coupling~\cite{johned}: For given data
$(f,u_0,\phi_0) \in L^2(\Omega) \times H^{1/2}(\Gamma) \times H^{-1/2}(\Gamma)$, find $\uu =
(u,\phi)\in \HH := H^1(\Omega) \times H^{-1/2}(\Gamma)$ such that
\begin{subequations}\label{eq:fb:varform}
\begin{align}
  \int_\Omega B\nabla u \cdot \nabla v \,d\Omega - \int_\Gamma \phi v \,d\Gamma &=
  \int_\Omega f v \,d\Omega + \int_\Gamma \phi_0 v \,d\Gamma, \\
  \int_\Gamma \big((1/2-\dlp)u + \slp \phi\big)\psi \,d\Gamma &= \int_\Gamma (1/2-\dlp)u_0 \psi\,d\Gamma
\end{align}
\end{subequations}
for all $\vv=(v,\psi)\in\HH$. Here, $\slp \psi(x) := \int_\Gamma G(x-y) \psi(y)\,d\Gamma(y)$ is the simple-layer integral
operator and $\dlp v(x) := \int_\Gamma \partial_{\normal(y)} G(x-y)
v(y) \,d\Gamma(y)$ is the double-layer integral operator, with $G$ being the fundamental solution~\eqref{eq:G} of the Laplacian.
To ensure ellipticity of $\slp:H^{-1/2}(\Gamma)\to H^{1/2}(\Gamma)=(H^{-1/2}(\Gamma))^*$, we assume $\diam(\Omega)<1$ for $d=2$ by scaling;
see also Section~\ref{section:symm}.
Let $\norm{\vv}\HH^2 := \norm{v}{H^1(\Omega)}^2 + \norm{\psi}{H^{-1/2}(\Gamma)}^2$ for $\vv = (v,\psi)\in\HH$ denote the
canonical product norm on $\HH$.

The left-hand side of~\eqref{eq:fb:varform} gives rise to some operator 
$A : \HH\to \HH^*$.
The right-hand side of~\eqref{eq:fb:varform} gives rise to some $F\in\HH^*$ which
depends on the given data $f,u_0,\phi_0$. Then, \eqref{eq:fb:varform} can equivalently be reformulated
by~\eqref{eq:galerkin} with $\XX_\star = \HH$.
Note that $\dual{\phi}{\psi}_\slp := \int_\Gamma \psi\slp\phi\,d\Gamma$ defines a scalar product on $H^{-1/2}(\Gamma)$ with induced norm
$\norm\cdot\slp^2:=\dual\cdot\cdot_\slp$.
The following proposition states that 
the FEM-BEM formulation~\eqref{eq:fb:varform} fits into the abstract frame of 
Section~\ref{section:abstract}.

\begin{proposition}\label{prop:fembem}
The operator $A : \HH\to \HH^*$ associated with the left-hand side 
of~\eqref{eq:fb:varform} is bi-Lipschitz continuous~\eqref{eq:cont},
where $\Ccont>0$ depends only on $\c{lipB}$, $\c{monB}$, and $\Omega$.
Let $F\in\HH^*$ and let $\XX_\star$ be a closed subspace of $\HH$.
Provided that $(0,1)\in\XX_\star$, i.e.\ $\XX_{00}={\rm span}\{(0,1)\}$,
the variational formulation~\eqref{eq:galerkin} 
admits a unique solution $\U_\star(F)=(U_\star(F),\Phi_\star(F))\in\XX_\star$, and there holds the
C\'ea lemma~\eqref{eq:cea}. The constant $\c{cea}>0$ depends only on 
$\c{lipB}$, $\c{monB}$, and $\Omega$.
\end{proposition}

\begin{proof}[Sketch of proof]
The statements on unique solvability and C\'ea-type quasi-optimality are proved 
in~\cite{affkmp:fembem}; see also~\cite{sayas09} for the linear Laplace transmission 
problem, where $B$ is the identity. It only remains to show that $A$ is bi-Lipschitz.
The upper bound in~\eqref{eq:cont} follows from Lipschitz 
continuity~\eqref{eq:fembem:lip} of $B$ and the continuity of the boundary 
integral operators 
$\slp: H^{-1/2}(\Gamma)\to H^{1/2}(\Gamma)$ and 
$\dlp : H^{1/2}(\Gamma) \to H^{1/2}(\Gamma)$.
For the lower bound in~\eqref{eq:cont}, we use the definition of the dual norm
  \begin{align*}
    \norm{A\uu-A\vv}{\HH^*} = \sup\limits_{\ww=(w,\chi)\in\HH\backslash\{(0,0)\}} \frac{|\dual{A\uu-A\vv}{\ww}|}{\norm{\ww}\HH}
  \end{align*}
 For $\uu=(u,\phi), \vv = (v,\psi) \in\HH$, we 
 choose $\ww = \uu-\vv + (0,1)\int_\Gamma (\tfrac12-\dlp)(u-v) + \slp(\phi-\psi)
  \,d\Gamma$. 
    By continuity of $\slp$ and $\dlp$, it follows $\norm{\ww}\HH \lesssim \norm{\uu-\vv}\HH$,
  where the hidden constant depends only on $\Omega$. Moreover,
$\ww=(0,0)$ implies that $u=v$ and $\phi-\psi=-\dual{\phi-\psi}{1}_\slp=:c\in\R$ is 
constant. With this identity, it follows $0=(1+\dual{1}{1}_\slp)c$.
Ellipticity of 
$\slp$ proves $0=c=\phi-\psi$, i.e., $\ww=0$ yields
  $\uu=\vv$.

  The theory of implicit stabilization provided in~\cite{affkmp:fembem} shows $\dual{A\uu-A\vv}{\ww} \gtrsim
  \norm{\uu-\vv}\HH^2$, where the hidden constant depends only on $\c{monB}$,
  and $\Omega$. For $\uu\neq\vv$, we altogether obtain 
 ${|\dual{A\uu-A\vv}{\ww}|}/{\norm{\ww}\HH} \geq \Ccont^{-1} \norm{\uu-\vv}\HH$, where $\Ccont>0$ depends only
  on $\c{lipB},\c{monB}$,
  and $\Omega$.
\end{proof}



\subsection{Discretization}
Let $\TT_\star^\Omega$ be a $\gamma$-shape regular triangulation of $\Omega$ into triangles for $d=2$
resp. tetrahedrons for $d=3$.
Here, $\gamma$-shape regularity means
\begin{align}
  \sup\limits_{T\in\TT_\star^\Omega} \frac{\diam(T)^d}{|T|} \leq \gamma < \infty
\end{align}
with $|\cdot|$ being the $d$-dimensional volume measure.
Suppose that $\TT_\star^\Omega$ is regular in the sense of Ciarlet, i.e., $\TT_\star^\Omega$ admits no hanging nodes.
Let $\TT_\star^\Gamma := \TT_\star^\Omega|_\Gamma$ be the triangulation of $\Gamma$ which is induced by
$\TT_\star^\Omega$. Note that $\TT_\star^\Gamma$ then is $\widetilde\gamma$-shape regular in the sense
of~\eqref{eq:shaperegular}, where $\widetilde\gamma>0$ depends only on $\gamma$.
Moreover, for $d=3$, $\TT_\star^\Gamma$ is regular in the sense of Ciarlet as well.
We formally consider $\TT_\star := \TT_\star^\Omega \cup \TT_\star^\Gamma$ with the abstract notation of
Section~\ref{section:abstract}.
Let $\SS^1(\TT_\star^\Omega)$ be the space of piecewise affine, globally continuous functions on $\TT_\star^\Omega$
and $\PP^0(\TT_\star^\Gamma)$ be the space of all $\TT_\star^\Gamma$-piecewise constant functions.
With $\XX_\star := \SS^1(\TT_\star^\Omega) \times \PP^0(\TT_\star^\Gamma)$, we now consider the Galerkin 
formulation~\eqref{eq:galerkin}.
The discrete solution with respect to $\XX_\star$ will be denoted by $\UU_\star = (U_\star,\Phi_\star)$.

\subsection{Weighted-residual error estimator}
Assume additional regularity $(f,u_0,\phi_0) \in L^2(\Omega) \times H^{1}(\Gamma) \times L^2(\Gamma)$.
Following~\cite{cs95:fembem}, it is proved in~\cite{afkp:fembem} for linear problems and 
in~\cite{affkmp:fembem} for strongly monotone problems that 
\begin{align}
  \norm{\uu(F) - \UU_\star(F)}\HH \simeq \norm{F-A\UU_\star(F)}{\HH^*} \lesssim \eta_\star (F),
\end{align}
where the error estimator $\eta_\star(F)^2 := \sum_{T\in\TT_\star} \eta_\star(F;T)^2$
is defined by
\begin{subequations}\label{eq:fembem:eta}
\begin{align}
\begin{split}
  \eta_\star(F;T)^2 &:= \diam(T)^2\, \norm{f}{L^2(T)}^2
  \\& +\! \diam(T)\, \Big(\norm{[B\nabla U_\star\cdot\normal]}{L^2(\partial T\backslash \Gamma)}^2
  +\norm{\phi_0+\Phi_\star - B\nabla U_\star\cdot\normal}{L^2(\partial T\cap\Gamma)}^2
  \Big)
\end{split}
\intertext{for $T\in\TT_\star^\Omega$ resp.}
  \eta_\star(F;T)^2 &:= \diam(T) \,\norm{ \nabla_\Gamma\big((1/2-\dlp)(U_\star-u_0) + \slp \Phi_\star\big)}{L^2(T)}^2 
\end{align}
\end{subequations}
for $T\in\TT_\star^\Gamma$.
Here, $[B\nabla U_\star \cdot\normal]$ denotes the jump of $B\nabla U \cdot \normal$ across interior facets $E$, where
$E = T_+\cap T_-$ for some $T_+,T_-\in\TT_\star^\Omega$ with $T_+\neq T_-$.
By means of the estimator reduction principle~\cite{afp}, it follows that Algorithm~\ref{algorithm} converges for
$\eta_\ell(F) = \mu_\ell(F)$; see~\cite{affkmp:fembem}.

\subsection{Two-level error estimator}
Two-level error estimators for the adaptive coupling of FEM and BEM have first been proposed in~\cite{ms:fembem}.
Let $\widehat\TT_\star^\Omega$ denote the uniform refinement of $\TT_\star^\Omega$. 
Let $\widehat\NN_\star^\Omega$ be the corresponding set of nodes
and $\widehat\TT_\star^\Gamma:=\widehat\TT_\star^\Omega|_\Gamma$ be the induced
triangulation of $\Gamma$.
For each element
$T\in\TT_\star^\Omega$, let $z_{T,j}\in T\cap\widehat\NN_\star^\Omega$, $j=1,\dots,D^\Omega$ denote the new nodes of the uniform refinement
$\widehat\TT_\star^\Omega$ within $T$.
Let $v_{T,j} \in \SS^1(\widehat\TT_\star^\Omega)$ denote the fine-mesh hat functions, which satisfy
$v_{T,j}(z_{T,j}) = 1$ and $v_{T,j}(z) = 0$ for all $z\in\widehat\NN_\star^\Omega\backslash\{z_{T,j}\}$.
Moreover, let $\{\chi_T,\psi_{T,j},\dots,\psi_{T,D^\Gamma}\}$ denote a basis of $\PP^0(\widehat\TT_\star^\Gamma|_T)$ for
each element $T\in\TT_\star^\Gamma$, with $\chi_T$ being the characteristic function on $T$ and $\int_\Gamma
\psi_{T,j} \,d\Gamma = 0$.
Then, the two-level estimator $\mu_\star^2 := \sum_{T\in\TT_\star} \mu_\star(F;T)^2$ is
defined by
\begin{subequations}\label{eq:fembem:twolevel}
\begin{align}
  \mu_\star(F;T)^2 := \sum_{j=1}^{D^\Omega} \mu_{\star,j}(F;T)^2 \quad\text{with}\quad \mu_{\star,j}(F;T) := \frac{\dual{F-A
  \UU_\star(F)}{(v_{T,j},0)}}{\norm{v_{T,j}}{H^1(\Omega)}}
\end{align}
for $T\in\TT_\star^\Omega$ and
\begin{align}
  \mu_\star(F;T)^2 := \sum_{j=1}^{D^\Gamma} \mu_{\star,j}(F;T)^2 \quad\text{with}\quad \mu_{\star,j}(F;T) := \frac{\dual{F-A
  \UU_\star(F)}{(0,\psi_{T,j})}}{\norm{\psi_{T,j}}\slp}
\end{align}
for $T\in\TT_\star^\Gamma$.
\end{subequations}
Note that unlike the weighted-residual error estimator~\eqref{eq:fembem:eta},
the two-level error estimator~\eqref{eq:fembem:twolevel} does not require additional regularity of the data,
but only $(f,u_0,\phi_0)\in L^2(\Omega)\times H^{1/2}(\Gamma)\times H^{-1/2}(\Gamma)$.

The two-level estimator $\mu_\star$ is known to be efficient
\begin{align}
  \mu_\star(F) \leq \c{eff} \norm{\uu(F)-\UU_\star(F)}{\HH},
\end{align}
while reliability
\begin{align}
  \norm{\uu(F)-\UU_\star(F)}{\HH} \leq \c{rel} \mu_\star(F)
\end{align}
holds under the saturation assumption
\begin{align}
  \norm{\uu(F)-\widehat\UU_\star(F)}\HH \leq q_{\rm sat} \norm{\uu(F)-\UU_\star(F)}\HH; 
\end{align}
see~\cite{afkp:fembem} for the linear Johnson-N\'ed\'elec coupling and
the seminal work~\cite{ms:fembem} for some non-linear symmetric coupling.
Here, $\widehat\UU_\star(F)$ denotes the Galerkin solution with respect to the uniform refinement
$(\widehat\TT_\star^\Omega,\widehat\TT_\star^\Gamma)$ of $(\TT_\star^\Omega,\TT_\star^\Gamma)$, and $0<q_{\rm sat}<1$ is a uniform constant.
The details are left to the reader.

\begin{theorem}\label{thm:fembem:twolevel}
Suppose that the two-level error estimator~\eqref{eq:fembem:twolevel} is used
for marking~\eqref{eq:doerfler}. Suppose that the mesh-refinement
guarantees uniform $\gamma$-shape regularity of the meshes $\TT_\ell^\Omega,\TT_\ell^\Gamma$ generated,
as well as that all marked elements $T\in\MM_\ell \subseteq \TT_\ell^\Omega \cup \TT_\ell^\Gamma$ are refined into sons 
$T'\in\TT_{\ell+1} = \TT_{\ell+1}^\Omega \cup \TT_{\ell+1}^\Gamma$ with $|T'|\le \kappa\,|T|$ with some uniform constant $0<\kappa<1$,
where $|\cdot|$ denotes the $d$-dimensional volume measure for $T\in\TT_\ell^\Omega$ resp.\ the $(d-1)$-dimensional surface measure for $T\in\TT_\ell^\Gamma$.
Then, Algorithm~\ref{algorithm} guarantees 
\begin{align}\label{fembem:convergence}
 \mu_\ell(F)\to0
 \quad\text{as }\ell\to\infty
\end{align}
for all $F\in \HH^*$.
\end{theorem}

Our proof of Theorem~\ref{thm:fembem:twolevel} requires the following two results, which essentially state stability of two-level
decompositions of the discrete space $\widehat\XX_\ell := \SS^1(\widehat\TT_\ell^\Omega)\times \PP^0(\widehat\TT_\ell^\Gamma)$.
The following lemma is a consequence of~\cite[Theorem 4.1]{harry86} 
and explicitly stated in~\cite[Lemma~3.1]{ms:fembem}. It provides a hierarchical splitting of $\SS^1(\widehat\TT_\ell^\Omega)$.

  \begin{lemma}\label{lemma:decomp:Omega}
  Let $\proj_\ell^\Omega : H^1(\Omega)\to \SS^1(\TT_\ell^\Omega)$ and $\proj_{T,j}^\Omega : H^1(\Omega) \to \linhull\{v_{T,j}\}$ denote the 
  $H^1$-orthogonal projections. 
  For $\widehat V_\ell\in\SS^1(\widehat\TT_\ell^\Omega)$, it then holds
   \begin{align}
    \c{decompOmega}^{-1} \norm{\widehat V_\ell}{H^1(\Omega)}^2 \leq \norm{\proj_\ell^\Omega \widehat V_\ell}{H^1(\Omega)}^2 +
    \sum_{T\in\TT_\ell^\Omega} \sum_{j=1}^{D^\Omega} \norm{\proj_{j,T}^\Omega \widehat V_\ell}{H^1(\Omega)}^2 
    \leq \c{decompOmega} \norm{\widehat V_\ell}{H^1(\Omega)}^2.
  \end{align}
  The constant $\setc{decompOmega}>0$ depends only on $\Omega$ and the $\gamma$-shape regularity of $\TT_\ell^\Omega$.
  \qed
\end{lemma}

The following lemma is found in~\cite[Proposition~4.5]{effp} and provides
a hierarchical splitting of $\PP^0(\widehat\TT_\ell^\Gamma)$. Although~\cite{effp} is only formulated for 2D BEM, the results and proofs hold verbatim for 3D.
(For 3D BEM and uniform meshes, the claim is already found in~\cite{msw}).

\begin{lemma}\label{lemma:decomp:Gamma}
  Let $\proj_\ell^\Gamma : H^{-1/2}(\Gamma) \to \PP^0(\TT_\ell^\Gamma)$ and
$\proj_{T,j}^\Gamma : H^{-1/2}(\Gamma) \to \linhull\{\psi_{T,j}\}$ denote the orthogonal projections with respect to the $\slp$-induced scalar product $\dual{\cdot}\cdot_\slp$ on $H^{-1/2}(\Gamma)$. For $\widehat \Psi_\ell\in\PP^0(\widehat\TT_\ell^\Gamma)$, it then holds
  \begin{align}
    \c{decompGamma}^{-1} \norm{\widehat \Psi_\ell}{\slp}^2 \leq \norm{\proj_\ell^\Gamma \widehat \Psi_\ell}{\slp}^2
    +  \sum_{T\in\TT_\ell^\Gamma} \sum_{j=1}^{D^\Gamma} \norm{\proj_{j,T}^\Gamma \widehat \Psi_\ell}{\slp}^2 
    \leq \c{decompGamma} \norm{\widehat \Psi_\ell}{\slp}^2.
  \end{align}
  The constant $\setc{decompGamma}>0$ depends only on $\Gamma$ and the $\gamma$-shape regularity of $\TT_\ell^\Gamma$.
  \qed
\end{lemma}
  
\begin{proof}[Proof of Theorem~\ref{thm:fembem:twolevel}]
  The proof is similar to the one of Theorem~\ref{thm:symm:twolevel} and relies on the verification
  of~\eqref{ass:local}--\eqref{ass:stable} to apply Proposition~\ref{prop:new}. 
  For patches, we use the notation~\eqref{eq:kpatch} from the proof of Theorem~\ref{thm:faermann-convergence}, but now
defined for volume elements, i.e., $\TT_\ell^\Omega$ instead of 
$\TT_\ell^\Gamma=\TT_\ell$ in~\eqref{eq:kpatch}.

\next   We define the equivalent mesh-size function $\widetilde h_\ell : \Omega \to \R$ as in~\eqref{eq:mannomann} in the proof of
  Theorem~\ref{thm:faermann-convergence}, but
  now for volume elements $T\in\TT_\ell^\Omega$, as 
  well as $\widetilde h_\ell(T) := |T|^{1/(d-1)}$ for boundary elements $T\in\TT_\ell^\Gamma$.
  The auxiliary estimator $\rho_\ell(F)^2:= \sum_{T\in\TT_\ell} \rho_\ell(F;T)^2$ is defined by
  \begin{subequations}
  \begin{align}
  \begin{split}
    \rho_\ell(F;T)^2 := &\norm{\widetilde h_\ell f}{L^2(T)}^2 + \norm{\widetilde h_\ell^{1/2} [B\nabla
    U_\ell\cdot\normal]}{L^2(\partial T\backslash\Gamma)}^2 \\&+\norm{\widetilde h_\ell^{1/2}(\phi_0 + \Phi_\ell - B\nabla
    U_\ell\cdot \normal)}{L^2(\partial T\cap\Gamma)}^2 
\end{split}
  \end{align}
  for volume elements $T\in\TT_\ell^\Omega$ and
  \begin{align}\label{eq:hassmich1}
    \rho_\ell(F;T)^2 := \norm{\widetilde h_\ell^{1/2}\nabla_\Gamma \big( (1/2-\dlp)(U_\ell-u_0) +
    \slp\Phi_\ell\big)}{L^2(T)}^2
  \end{align}
  \end{subequations}
  for boundary elements $T\in\TT_\ell^\Gamma$. We note that $\eta_\ell(F;T) \simeq \rho_\ell(F;T)$ for all $T\in\TT_\ell^\Omega\cup\TT_\ell^\Gamma$, where the 
  hidden constants
  depend only on the $\gamma$-shape regularity of $\TT_\ell^\Omega$.

\next To prove~\eqref{ass:local}, we proceed similar to the proof 
of~\cite[Theorem~12]{afkp:fembem}.
   Let $T\in\TT_\ell^\Omega$. 
  Denote by $\EE_\ell^\Omega(z_{T,j})$ all interior facets of the patch 
  $\omega_\ell(z_{T,j}):=\set{T'\in\TT_\ell^\Omega}{z_{T,j}\in T'}\subseteq\TT_\ell^\Omega$.
  Piecewise integration by parts shows
  \begin{align*}
    \dual{F-A\UU_\ell(F)}{(v_{T,j},0)} &= \int_\Omega fv_{T,j}\,d\Omega + \int_\Gamma (\phi_0+\Phi_\ell)v_{T,j} \,d\Gamma - \int_\Omega
    B\nabla U_\ell \cdot\nabla v_{T,j} \,d\Omega \\
    &= \sum_{ T'\in \omega_\ell(z_{T,j})} \Big( \int_{T'} fv_{T,j}\,d\Omega + \int_{\Gamma\cap\partial
    T'} (\phi_0+\Phi_\ell-B\nabla U_\ell\cdot\normal) v_{T,j} \,d\Gamma \Big)\\
    &\qquad\qquad - \sum_{E\in\EE_\ell^\Omega(z_{T,j})} \int_{E} 
    [B\nabla U_\ell\cdot\normal]v_{T,j} \,dE,
  \end{align*}
  where we have used that $\mathrm{div}B\nabla U_\ell = 0$ on each element $T\in\TT_\ell^\Omega$.
  Note that 
  $$\diam(T)\,\norm{\nabla v_{T,j}}{L^2(\Omega)}\simeq\norm{v_{T,j}}{L^2(\Omega)}\simeq\diam(T)^{d/2}.$$
and consequently also
  $\norm{v_{T,j}}{L^2(E)} \lesssim \diam(T)^{(d-1)/2}$ for each facet $E\subseteq T$.
  For the volume contributions of the two-level estimator, this yields the estimate
  \begin{align*}
    \mu_{\ell,j}(F;T)^2 &\lesssim \diam(T)^2 \norm{f}{L^2(\omega_\ell(z_{T,j}))}^2 +
    \sum_{ T'\in \omega_\ell(z_{T,j})}\diam(T)\,
    \norm{[B\nabla U_\ell \cdot\normal]}{L^2(\partial T'\backslash\Gamma)}^2 \\
    &\qquad+ \sum_{ T'\in \omega_\ell(z_{T,j})}\diam(T) \,\norm{\phi_0 + \Phi_\ell -B\nabla
    U_\ell\cdot\normal}{L^2(\partial T'\cap \Gamma)}^2 \\
    &\lesssim \eta_\ell(F; \omega_\ell(z_{T,j}))^2 
    \simeq \rho_\ell(F; \omega_\ell(z_{T,j}))^2. 
  \end{align*}
  The contribution $\mu_\ell(F;T)$ of the two-level estimator for boundary elements $T\in\TT_\ell^\Gamma$ coincides essentially with the
    two-level estimator~\eqref{eq:symm:twolevel} of Section~\ref{section:symm}, and $\eta_\ell(F;T)$ coincides essentially with the corresponding definition~\eqref{eq:resmean} in
    Section~\ref{section:symm}.
    Arguing along the lines of Theorem~\ref{thm:symm:twolevel}, we hence obtain for
    each boundary element $T\in\TT_\ell^\Gamma$
  \begin{align*}
    \mu_{\ell,j}(F;T)^2 \lesssim \eta_{\ell}(F;T)^2
    \simeq \rho_{\ell}(F;T)^2.
  \end{align*}
  Summing over all $j$ and $T\in\MM_\ell = \MM_\ell^\Omega \cup \MM_\ell^\Gamma \subseteq
  \TT_\ell^\Omega\cup\TT_\ell^\Gamma$,
  we prove assumption~\eqref{ass:local}
   with $\RR_\ell = \RR_\ell^\Omega \cup \RR_\ell^\Gamma 
  = \omega_\ell(\MM_\ell^\Omega) \cup \MM_\ell^\Gamma$.

  \next
  For the verification of~\eqref{ass:contraction} we proceed similar to the proof of Theorem~\ref{thm:symm:twolevel} and
  Theorem~\ref{thm:faermann-convergence}.
  Each contribution of the estimator $\rho_\ell(F)$ can be estimated separately.

  First, note that $\widetilde h_{\ell+1}|_{\bigcup \RR_\ell^\Omega} \leq q \widetilde
  h_\ell|_{\bigcup\RR_\ell^\Omega}$ for the constant $0<q<1$ from~\eqref{eq:mannomann}. Therefore, $\norm{\widetilde h_{\ell+k} f}{L^2(\bigcup
  \RR_\ell^\Omega)}^2 \leq q^2 \norm{\widetilde h_\ell f}{L^2(\bigcup\RR_\ell^\Omega)}^2$, and we further obtain
  \begin{align*}
    (1-q^2) \norm{\widetilde h_\ell f}{L^2(\bigcup\RR_\ell^\Omega)}^2 \leq \norm{\widetilde h_\ell f}{L^2(\Omega)}^2 -
    \norm{\widetilde h_{\ell+k} f}{L^2(\Omega)}^2.
  \end{align*}

  Second, note that $\widetilde h_\ell - \widetilde h_{\ell+k} \geq (1-q) \widetilde h_{\ell}$ on
  $\bigcup\RR_\ell^\Omega$. We estimate
  \begin{align*}
    &(1-q) \sum_{T\in\RR_\ell^\Omega} \norm{\widetilde h_\ell^{1/2} [B\nabla U_\ell\cdot\normal]}{L^2(\partial
    T\backslash\Gamma)}^2 \leq \sum_{T\in\RR_\ell^\Omega} \norm{(\widetilde h_\ell-\widetilde h_{\ell+k})^{1/2}
    [B\nabla U_\ell\cdot\normal]}{L^2(\partial T\backslash\Gamma)}^2 \\
    &\qquad \leq \sum_{T\in\TT_\ell^\Omega} \norm{\widetilde h_\ell^{1/2}
    [B\nabla U_\ell\cdot\normal]}{L^2(\partial T\backslash\Gamma)}^2 -
    \sum_{T\in\TT_\ell^\Omega} \norm{\widetilde h_{\ell+k}^{1/2}  [B\nabla U_\ell\cdot\normal]}{L^2(\partial
    T\backslash\Gamma)}^2.
  \end{align*}
  For the second term, we note that the jumps $[B\nabla U_\ell\cdot\normal]$ across newly created facets in
  $\TT_{\ell+k}^\Omega$ vanish. Hence, $\sum_{T\in\TT_\ell^\Omega} \norm{\widetilde h_{\ell+k}^{1/2}  [B\nabla U_\ell\cdot\normal]}{L^2(\partial
  T\backslash\Gamma)}^2 = \sum_{T\in\TT_{\ell+k}^\Omega} \norm{\widetilde h_{\ell+k}^{1/2}  [B\nabla U_\ell\cdot\normal]}{L^2(\partial
  T\backslash\Gamma)}^2$.
  The triangle inequality and Young's inequality yield, for all $\delta>0$, 
  \begin{align*}
    \sum_{T\in\TT_{\ell+k}^\Omega} \norm{\widetilde h_{\ell+k}^{1/2}  [B\nabla U_\ell\cdot\normal]}{L^2(\partial
    T\backslash\Gamma)}^2 &\leq  (1\!+\!\delta) \! \sum_{T\in\TT_{\ell+k}^\Omega} \norm{\widetilde h_{\ell+k}^{1/2}[B\nabla
    U_{\ell+k}\cdot\normal]}{L^2(\partial T\backslash\Gamma)}^2 \\ &\quad + (1\!+\!\delta^{-1}) \!
    \sum_{T\in\TT_{\ell+k}^\Omega} \norm{\widetilde h_{\ell+k}^{1/2}[(B\nabla
    U_\ell\!-\!B\nabla U_{\ell+k})\cdot\normal]}{L^2(\partial T\backslash\Gamma)}^2.
  \end{align*}
  A scaling argument and Lipschitz continuity of $B$ show
   that $\sum_{T\in\TT_{\ell+k}^\Omega} \norm{\widetilde h_{\ell+k}^{1/2}
  [(B\nabla U_\ell-B\nabla U_{\ell+k})\cdot\normal]}{L^2(\partial T\backslash \Gamma)}^2 \leq \Cinv
  \norm{U_\ell-U_{\ell+k}}{H^1(\Omega)}^2$.  The constant $\Cinv>0$ depends only on $\c{lipB}$ and $\gamma$-shape regularity of $\TT_\ell^\Omega$.
  Details can
  be found, e.g., in the proof of~\cite[Theorem~15]{afembem}.
  Arguing as in the proof of Theorem~\ref{thm:symm:twolevel}, we
 obtain
  \begin{align*}
    &(1-q) \sum_{T\in\RR_\ell^\Omega} \norm{\widetilde h_\ell^{1/2} [B\nabla U_\ell\cdot\normal]}{L^2(\partial
    T\backslash\Gamma)}^2 \\
    &\qquad\leq \sum_{T\in\TT_\ell^\Omega} \norm{\widetilde h_\ell^{1/2} [B\nabla
    U_\ell\cdot\normal]}{L^2(\partial T\backslash\Gamma)}^2 - \frac1{1+\delta} \sum_{T\in\TT_{\ell+k}^\Omega}
    \norm{\widetilde h_{\ell+k}^{1/2} [B\nabla U_{\ell+k}\cdot\normal]}{L^2(\partial T\backslash\Gamma)}^2 \\
    &\hspace*{6cm} + \frac{1+\delta^{-1}}{1+\delta} \Cinv \norm{U_\ell-U_{\ell+k}}{H^1(\Omega)}^2.
  \end{align*}
 
 Third, similar arguments as before yield
  \begin{align*}
    &(1-q) \sum_{T\in\RR_\ell^\Omega} \norm{\widetilde h_\ell^{1/2}(\phi_0 + \Phi_\ell - B\nabla
    U_\ell\cdot\normal)}{L^2(\partial T\cap \Gamma)}^2 \\
    &\leq \sum_{T\in\TT_\ell^\Omega} \norm{\widetilde h_\ell^{1/2} (\phi_0 + \Phi_\ell - B\nabla
    U_\ell\cdot\normal)}{L^2(\partial T\cap\Gamma)}^2 \\ &\quad- \frac1{1+\delta} \sum_{T\in\TT_{\ell+k}^\Omega}
    \norm{\widetilde h_{\ell+k}^{1/2} (\phi_0 + \Phi_{\ell+k} - B\nabla U_{\ell+k}\cdot\normal)}{L^2(\partial T\cap\Gamma)}^2 
    + \frac{1+\delta^{-1}}{1+\delta} \Cinv \norm{\UU_\ell-\UU_{\ell+k}}{\HH}^2.
  \end{align*}

  Fourth, note that $\rho_\ell(F;T)$ for boundary elements $T\in\TT_\ell^\Gamma$ is similarly defined as in
  the proof of Theorem~\ref{thm:symm:twolevel}.
  Therefore, the contraction of the BEM contribution $\rho_\ell(F;\RR_\ell^\Gamma)$ from~\eqref{eq:hassmich1} follows with the same arguments as in the proof of
  Theorem~\ref{thm:symm:twolevel}. In addition to the inverse estimate~\eqref{eq:novel:invest} for the simple-layer integral operator $\slp$, we require a similar estimate for the double-layer
  integral operator
  \begin{align}\label{eq:invest:dlp}
   \norm{h_\ell^{1/2}\nabla_\Gamma(1/2-\dlp)U_\ell}{L^2(\Gamma)}
   \lesssim \norm{U_\ell}{H^{1/2}(\Gamma)},
  \end{align}
  which is also provided by~\cite[Corollary~3]{afembem}.
  
Combining the last four steps, we prove assumption~\eqref{ass:contraction}.

  \next
  For the last assumption~\eqref{ass:stable}, the definition of $\mu_\ell$ from~\eqref{eq:fembem:twolevel} shows
  \begin{align}\label{eq:proof:fembem:a4:1}
  \begin{split}
    |\mu_\ell(F;\MM_\ell)-\mu_\ell(F';\MM_\ell)|^2 &\leq \sum_{T\in\TT_\ell^\Omega} \sum_{j=1}^{D^\Omega} 
    \frac{\dual{F-F'-(A\UU_\ell(F)-A\UU_\ell(F'))}{(v_{T,j},0)}^2}{\norm{v_{T,j}}{H^1(\Omega)}^2} \\
    &\quad+ \sum_{T\in\TT_\ell^\Gamma} \sum_{j=1}^{D^\Gamma} 
    \frac{\dual{F-F'-(A\UU_\ell(F)-A\UU_\ell(F'))}{(0,\psi_{T,j})}^2}{\norm{\psi_{T,j}}\slp^2}.
  \end{split}
  \end{align}
 Define the scalar product
  \begin{align*}
    \edual{\uu}{\vv} := \int_\Omega \nabla u\cdot\nabla v \,d\Omega + \int_\Omega uv\,d\Omega + \dual{\phi}{\psi}_\slp
  \end{align*}
  for all $\uu=(u,v),\vv=(v,\psi)\in\HH$ with induced norm $\enorm{\cdot}^2=\edual\cdot\cdot$. By the Riesz theorem, there exists a unique
  $\widehat\WW_\ell = (W_\ell,\Xi_\ell)\in \widehat\XX_\ell$ with
  \begin{align*}
    \edual{\widehat\WW_\ell}{\widehat\VV_\ell} = \dual{F-F'-(A\UU_\ell(F)-A\UU_\ell(F'))}{\widehat\VV_\ell}
  \end{align*}
  Let $\proj_\ell : \HH\to\XX_\ell$ with $\proj_\ell\vv := (\proj_\ell^\Omega v,\proj_\ell^\Gamma \psi)$ for all $\vv=(v,\psi)\in\HH$.
  Together with symmetry of the orthogonal projection $\proj_\ell$, the last identity and the Galerkin orthogonality prove
  \begin{align*}
    \enorm{\proj_\ell \widehat\WW_\ell}^2 = \edual{\proj_\ell \widehat\WW_\ell}{\proj_\ell \widehat\WW_\ell} =
    \edual{\widehat\WW_\ell}{\proj_\ell \widehat\WW_\ell} = 0.
  \end{align*}
  From Lemma~\ref{lemma:decomp:Omega} and Lemma~\ref{lemma:decomp:Gamma}, it thus follows
  \begin{align*}
    \enorm{\widehat\WW_\ell}^2 \simeq \sum_{T\in\TT_\ell^\Omega} \sum_{j=1}^{D^\Omega} \norm{\proj_{T,j}^\Omega \widehat W_\ell}{H^1(\Omega)}^2
    +  \sum_{T\in\TT_\ell^\Gamma}\sum_{j=1}^{D^\Gamma} \norm{\proj_{T,j}^\Gamma \widehat \Xi_\ell}{\slp}^2.
  \end{align*}
  We stress that the last term is equal to the right-hand side of~\eqref{eq:proof:fembem:a4:1} and proceed by using the Lipschitz continuity
  of $A$ to estimate
  \begin{align*}
  |\mu_\ell(F;\MM_\ell)-\mu_\ell(F';\MM_\ell)| \lesssim
    \enorm{\widehat\WW_\ell} 
    &= \norm{F-F'-(A\UU_\ell(F)-A\UU_\ell(F'))}{\widehat\XX_\ell^*} 
    \\&\lesssim \norm{F-F'}{\HH^*} +
    \norm{\UU_\ell(F)-\UU_\ell(F')}{\HH}.
  \end{align*}
  Arguing along the lines of Proposition~\ref{prop:fembem}, one proves that $A$ is even
  bi-Lipschitz continuous with
   respect to the discrete dual space $\XX_\ell^*$, i.e., $\norm{\VV_\ell-\widetilde\VV_\ell}\HH
  \simeq \norm{A\VV_\ell-A\widetilde\VV_\ell}{\XX_\ell^*}$ for all $\VV_\ell,\widetilde\VV_\ell\in\XX_\ell$. 
  Therefore, we get
  \begin{align*}
    \norm{\UU_\ell(F)-\UU_\ell(F')}\HH \simeq \norm{A\UU_\ell(F)-A\UU_\ell(F')}{\XX_\ell^*} = \norm{F-F'}{\XX_\ell^*}\leq
    \norm{F-F'}{\HH^*}.
  \end{align*}

  Altogether, we see
  \begin{align*}
    |\mu_\ell(F;\MM_\ell)-\mu_\ell(F';\MM_\ell)| \lesssim \enorm{\widehat\WW_\ell} \lesssim \norm{F-F'}{\HH^*},
  \end{align*}
  which proves assumption~\eqref{ass:stable}.
  
\end{proof}

\subsection{Remarks and extensions}
Although this section focused on the Johnson-N\'ed\'elec coupling~\cite{johned}, the same results 
hold also for the symmetric coupling~\cite{costabel} and the one-equation Bielak-MacCamy coupling~\cite{bmc}. We refer to~\cite{cs95:fembem} for the symmetric coupling in the presence of
strongly monotone nonlinearities and the first introduction of the corresponding 
weighted-residual error estimator and to~\cite{ms} for the corresponding two-level
estimator. 

In~\cite{cs95:fembem}, the analysis, based on the discrete (symmetric) Steklov-Poincar\'e operator, required the additional assumption that the initial boundary mesh $\TT_0^\Gamma$ is sufficiently fine. This assumption has first been proved to be unnecessary 
in~\cite{hh2fembem}, where the original argument of~\cite{cs95:fembem} is refined.
We note that even the extended argument is restricted to the symmetric Steklov-Poincar\'e operator
and thus only applies to the symmetric coupling.
The method of implicit stabilization from~\cite{affkmp:fembem} provides an alternate
proof of this fact which also transfers to the Johnson-N\'ed\'elec as well as the
Bielak-MacCamy coupling, i.e., no assumption on $\TT_0^\Gamma$ is required.

For the Bielak-MacCamy coupling, well-posedness of the coupling formulation in the presence of strongly monotone nonlinearities has
first been proved in~\cite{affkmp:fembem}, where also the corresponding weighted-residual error estimator is derived. The derivation of the corresponding 
two-level error estimator is not found in the literature yet, but is easily obtained
by adapting the arguments of, e.g., \cite{ms,afkp:fembem}.

Finally, we note that we only restricted to the lowest-order case $\XX_\star = \SS^p(\TT_\star^\Omega)\times\PP^{p-1}(\TT_\star^\Gamma)$ with $p=1$ for the ease
of presentation. All results also hold accordingly for higher order $p\ge1$.


\bigskip

\noindent
{\bf Acknowledgement.}
The research of MF, TF, GME, and DP is supported by the Aus-
trian Science Fund (FWF) through the research project {\em Adaptive boundary element method} funded under grant P21732
and the research project {\em Optimal adaptivity for BEM and FEM-BEM coupling}
funded under grant P27005. In addition, TF acknowledges support through
the {\em Innovative projects initiative} of Vienna University of Technology
(TU Wien), and MF and DP acknowledge support through the FWF doctoral program {\em Dissipation and dispersion in nonlinear PDEs} funded under
grant W1245.


\bibliographystyle{alpha}
\bibliography{literature}

\end{document}